\documentclass[12pt]{article}
\usepackage{amssymb}
\usepackage{mathrsfs}
\usepackage{amsfonts}
\usepackage{amsmath}
\usepackage{amsthm}
\usepackage{amssymb,url}
\usepackage{graphicx}
\usepackage{color}
\usepackage{ifpdf}
\usepackage{epstopdf}
\usepackage{cite}
\usepackage{indentfirst}
\newtheorem{thm}{Theorem}[section]

\newtheorem{lem}[thm]{Lemma}
\newtheorem{pro}[thm]{Proposition}
\newtheorem{cor}[thm]{Corollary}

\theoremstyle{definition}

\newtheorem{rem}{Remark}[section]

\makeatletter\@addtoreset{equation}{section}
\renewcommand\theequation
{\ifnum \c@section>\z@ \thesection.\fi\@arabic\c@equation}
\makeatother

\begin{document}

\title{Affine Orlicz P\'olya-Szeg\"o principles and their equality cases}

\author{Youjiang Lin \footnote{Research of the first named author is supported by the funds of the Basic and Advanced Research Project of CQ CSTC cstc2015jcyjA00009, cstc2018jcyjAX0790 and Scientific and Technological Research Program of Chongqing Municipal Education Commission (KJ1500628). }
\;and  Dongmeng Xi\footnote{Research of the second named author is supported by Shanghai Sailing Program 16YF1403800,
NSFC 11601310, and Chinese Post-doctoral Innovation Talent Support Program BX201600035.}}

\date{}
 \maketitle
{\bf Abstract}. The conjecture about the Orlicz P\'olya-Szeg\"o principle posed in \cite{LYJ17} is proved. The cases of equality are characterized in the affine Orlicz  P\'olya-Szeg\"o principle with respect to Steiner symmetrization and Schwarz spherical symmetrization.



\section{Introduction}

The classical P\'olya-Szeg\"o principle states that the $L_p$ norm of the gradient of any real-valued function from a certain class does not increase under an appropriate rearrangement. Schwarz spherical symmetrization about a point and Steiner symmetrization about a hyperplane are probably the most popular symmetrizations in the literature. P\'olya-Szeg\"o inequalities for these symmetrizations play a fundamental
role in the solution to a number of variational problems in different areas such as isoperimetric
inequalities, optimal forms of Sobolev inequalities, and sharp a priori estimates of solutions to
second-order elliptic or parabolic boundary value problems; see, for example, \cite{Ci00,CEFT08,CF02,CF06,ET04,FV04,Ma11,PS51} and
the references therein. In recent years, many important generalizations and variations have been obtained (see, e.g.,\cite{BCFP,BF15,Capriani14,Nguyen16,Wang13})

It is a remarkable discovery of Zhang \cite{Zhang99} that the Petty projection inequality, extended to a suitable class of nonconvex sets, can replace the isoperimetric inequality and then leads to an affine Sobolev inequality
which is stronger than the classical Sobolev inequality. An important ingredient in the proofs of sharp affine Sobolev-type inequalities is a strengthened affine P\'olya-Szeg\"o principle. The affine P\'olya-Szeg\"o principle asserts that
\begin{eqnarray}
\mathcal{E}_p(f^{\star})\leq\mathcal{E}_p(f),
\end{eqnarray}
where $f^{\ast}$ denotes the Schwarz spherical
symmetrization of $f$ and $\mathcal{E}_p$ denotes the $L^p$ affine energy of $f$ (see \cite{LYZ02}, or take $\phi(t)=|t|^p$ in  \eqref{1a}).
 It was proved by Lutwak, Yang and Zhang \cite{LYZ02} for $1\leq p<n$
and by Cianchi, Lutwak, Yang and Zhang \cite{CLYZ09} for all $p\geq 1$. In this remarkable affine rearrangement
inequality, an $L^p$ affine energy replaces the standard $L^p$ norm of the gradient leading
to an inequality which is significantly stronger than its classical Euclidean counterpart.
Moreover, Lutwak et al. \cite{LYZ02} and Cianchi et al. \cite{CLYZ09} obtained new sharp affine
Sobolev, Moser-Trudinger and Morrey-Sobolev inequalities by applying their affine
P\'olya-Szeg\"o principle, thereby demonstrating the power of this new affine symmetrization
inequality. Later, Haberl, Schuster, and Xiao \cite{HSX12} proved a remarkable asymmetric version of the affine P\'olya-Szeg\" o-type inequality which strengthens and implies the affine P\'olya-Szeg\" o principle of Cianchi et al. \cite{CLYZ09}. About the affine isoperimetric inequalities and their functional versions, also see \cite{CCF05,CPS15,FMP08,Gardner02,GZ98,HP14,HS0902,LX17,LL10,Ludwig10,LR99,LR10,LXZ11,LYZ00,LYZ06,XGL,XL16}.

The affine $L_p$ P\'olya-Szeg\"o-type principle is closely related to the $L_p$ Brunn-Minkowski theory of convex bodies (see, e.g., \cite{Lu93,Lu96,BLYZ12}). Based on the the seminal work of Lutwak, Yang and Zhang \cite{LYZ10,LYZ1002}, now the $L_p$ Brunn-Minkowski theory has been extended to the Orlicz-Brunn-Minkowski theory. The Orlicz-Brunn-Minkowski theory has expanded rapidly (see e.g., \cite{Boroczky13,GHW14,GHW15,HLYZ10,HP14,LYZ10,LYZ1002,XJL14,ZZX14,Zhu12}). It is natural to consider the affine P\'olya-Szeg\" o-type principle in Orlicz-Sobolev spaces. In \cite{LYJ17}, using functional Steiner symmetrization, the first named author proved an affine Orlicz P\'olya-Szeg\"o principle for log-concave functions, which includes the affine $L_p$ P\'olya-Szeg\"o  principle as special case. The case of equality of the  affine Orlicz P\'olya-Szeg\"o principle for log-concave functions is also characterized. In \cite{LYJ17}, the first named author conjectured that the principle can be extended to the general Orlicz-Sobolev functions.  In this paper, we confirm this conjecture and characterize the case of equality. An
affine Orlicz P\'olya-Szeg\"o principle with respect to Orlicz-Sobolev functions is formulated and proved.

In this paper, we mainly make use of Steiner
symmetrization of one-dimensional restrictions of Sobolev functions and Fubini's theorem to prove our results.  The proof  has the advantage of providing us with information about
functions yielding equality. The technique exploited in this paper differs from those of previous papers on affine P\'olya-Szeg\" o type inequalities, that make substantial use of fine results from the Brunn-Minkowski theory of convex bodies. The proof of the symmetric affine P\'olya-Szeg\" o principle in \cite{CLYZ09} mainly  relies on the $L_p$ Petty projection inequality from \cite{LYZ00}  and  the solution of the normalized $L_p$ Minkowski problem \cite{LYZ04}. The proof of the asymmetric affine P\'olya-Szeg\" o principle \cite{HS0902} mainly relies on a generalization of the $L_p$ Petty projection inequality established by Haberl and Schuster \cite{HS09} and  the solution of the normalized $L_p$ Minkowski problem \cite{LYZ04}. The techniques for proving the affine $L_p$ P\'olya-Szeg\" o principle could not be adapted to establish the affine Orlicz P\'olya-Szeg\" o principle. One of the reasons is that the function $\phi$ defining the Orlicz-Sobolev spaces is usually  not multiplicative, i.e., $\phi(xy)\neq\phi(x)\phi(y)$ for $x,y\in\mathbb{R}$. Moreover the Orlicz Minkowski problem has not been completely solved. Our approach is based on the functional Steiner symmetrization and makes use of a result for Steiner symmetrization with approximation of Schwarz
symmetrization by sequences of Steiner symmetrizations. Moreover, we  prove the affine Orlicz P\'olya-Szeg\"o  principle and its case of equality  not only with respect to Schwarz spherical symmetrization  but also with respect to Steiner symmetrization. The affine Orlicz P\'olya-Szeg\"o  principle  for Steiner symmetrization is new even in the $L^p$ setting.

In the remarkable paper \cite{CF06}, Cianchi and Fusco proved a beautiful P\'olya-Szeg\"o-type inequality  and analyzed the cases of equality  in Steiner symmetrization inequalities for Dirichlet-type
integrals.  The ideas and techniques of Cianchi and Fusco play a critical role
throughout this paper, especially in the proofs of the affine Orlicz P\'olya-Szeg\"o principle with respect to Steiner symmetrization and  its case of quality. It would be impossible to overstate our reliance on their work.

As pointed out in \cite{CF06}, investigations on the cases of equality in P\'olya-Szeg\"o type principles are more recent, and typically require an additional delicate analysis.
Such
a description has first been the object of the series of papers \cite{AFP00,BZ88,Ci10,CCF05,CF06,FV04,FMP08,Nguyen16} and has been recently
extended and simplified by new contributions, including \cite{BCFP,Bu97,Bu04,BF15,CEFT08,CF0602,ET04,ER09,Wang13}.
An impulse to
the study of this delicate issue was given by the paper \cite{Kawohl86}, where the
symmetry of PS-extremals for Schwarz and Steiner symmetrizations was established,
by classical techniques, in special classes of functions and domains.

In \cite{BZ88} Brothers and Ziemer characterized the equality cases in the
P\'olya-Szeg\"o inequality for the Schwarz rearrangement of a Sobolev function under the minimal
assumption that the set of critical points of the rearranged function has zero Lebesgue measure
(see also \cite{FV04} for an interesting alternate
proof). A version of
this result in the framework of functions of bounded variation can be found in \cite{CF02}.  A Brothers-Ziemer type theorem for
the affine P\'olya-Szeg\"o principle and a quantitative affine P\'olya-Szeg\"o principle were established by Wang \cite{Wang13}. The main goal of the second part of this paper is to prove a Brothers-Ziemer type theorem for the affine Orlicz P\'olya-Szeg\"o principle with respect to Schwarz symmetrization. Since the approach exploited in the paper relies on the result dealing with cases of equality for Steiner symmetrization, we assume that the domain of function is of finite perimeter. In view of  the available result for the Euclidean P\'olya-Szeg\"o principle, our assumption that the domain of function is a set of finite perimeter is probably unnecessary. However, if we remove such an assumption, then this would require the use of a different method to prove our result, that would go beyond the scope of the paper.

The paper is organized as follows. In Section \ref{s2} we state and comment the main results
and in Section \ref{s3} we collect some background material on  the Brunn-Minkowski theory and the theory of Sobolev functions. Section \ref{s4} is devoted to the affine Orlicz P\'olya-Szeg\"o principle with respect to Steiner symmetrization  while Section \ref{s5} deals with the case of Schwarz spherical symmetrization.

\section{Main results}\label{s2}
We begin with some definitions and elementary facts about Steiner symmetrization of sets and functions. Steiner symmetrization is a classical and very well-known
device, which has seen a number of remarkable applications to problems of
geometric and functional nature, see, e.g., \cite{Bianchi17,Bu96,Bu97,Bu04,CCF05,Lin17,Ta94,Trudinger97}.

Given two sets $E$ and $F$, we denote the {\it symmetric difference} by $E\triangle F:=(E\cup F)\setminus(E\cap F)$.
Given two open sets $\Omega^{\prime} \subset \Omega$ we write $\Omega^{\prime}\Subset\Omega$ if $\Omega^{\prime}$ is compactly contained in $\Omega$, i.e., ${\rm cl}\;\Omega^{\prime}\subset\Omega$, here ${\rm cl}\;\Omega^{\prime}$ denotes the closure of $\Omega^{\prime}$. A point $x$ in the Euclidean space $\mathbb{R}^n$, $n\geq 2$, will be usually labeled by $(x^{\prime},y)$, where $x^{\prime}=(x_1,\dots,x_{n-1})\in\mathbb{R}^{n-1}$ and $y\in\mathbb{R}$; similarly, when $x\in\mathbb{R}^{n+1}$, we shall write $x$ as $(x^{\prime},y,t)$. To emphasize the different roles of the variables $y$ and $t$, we shall also write $\mathbb{R}^n=\mathbb{R}^{n-1}\times\mathbb{R}_y$ and $\mathbb{R}^{n+1}=\mathbb{R}^{n-1}\times\mathbb{R}_y\times\mathbb{R}_t$. Consistent notations will be used for subsets of $\mathbb{R}^n$ and $\mathbb{R}^{n+1}$.  Let $\mathcal{L}^m$ denote the outer Lebesgue measure in $\mathbb{R}^m$. Throughout this paper, for $A,B\subset\mathbb{R}^n$, $A$ is {\it equivalent} to $B$ means that $\mathcal{L}^n(A\triangle B)=0$.

Given any measurable subset $E$ of $\mathbb{R}^n$, define, for $x^{\prime}\in\mathbb{R}^{n-1}$,
\begin{eqnarray}\label{2aa}
E_{x^{\prime}}=\{y\in\mathbb{R}:(x^{\prime},y)\in E\}
\end{eqnarray}
and
\begin{eqnarray}\label{2cc}
\ell_E(x^{\prime})=\mathcal{L}^1(E_{x^{\prime}}).
\end{eqnarray}
 Then, we define the {\it Steiner symmetral} $E^s$ of $E$ about the hyperplane $\{y=0\}$ as
$$E^s=\{(x^{\prime},y)\in\mathbb{R}^n:|y|<\ell_E(x^{\prime})/2\}.$$
When $E\subset\mathbb{R}^{n-1}\times\mathbb{R}_y\times\mathbb{R}_t$, its Steiner symmetral $E^s$ about $\{y=0\}$ is defined analogously, after replacing (\ref{2aa}) and (\ref{2cc}) by corresponding definitions of $E_{x^{\prime},t}$ and $\ell_E(x^{\prime},t)$.

Let $\pi_{n-1}(\Omega)$ denote the orthogonal projection of $\Omega\subset\mathbb{R}^n$ onto $\mathbb{R}^{n-1}$. Let $\Omega$ be a measurable subset of $\mathbb{R}^n$ and let $f$ be a nonnegative measurable function in $\Omega$ such that, for $\mathcal{L}^{n-1}$-a.e. $x^{\prime}\in\pi_{n-1}(\Omega)$,
\begin{eqnarray}\label{2g}
\mathcal{L}^1(\{y\in\Omega_{x^{\prime}}:f(x^{\prime},y)>t\})<\infty\;\;\;{\rm for\;every}\;t>0.
\end{eqnarray}
The {\it Steiner rearrangement} $f^s$ of $f$ is the function from $\mathbb{R}^n$ to $[0,+\infty]$ given by
$$f^s(x^{\prime},y)=\inf\{t>0:\;\mu_f(x^{\prime},t)\leq 2|y|\}\;\;{\rm for}\;(x^{\prime},y)\in\mathbb{R}^{n-1}\times\mathbb{R}_y,$$
where
$$\mu_f(x^{\prime},t)=\mathcal{L}^1(\{y\in\mathbb{R}:\;f_0(x^{\prime},y)>t\}),$$
the {\it distribution function} of $f(x^{\prime},\cdot)$, and $f_0$ denotes the continuation of $f$ to $\mathbb{R}^n$ which vanishes outside $\Omega$. Note that $f^s=0$ $\mathcal{L}^n$-a.e. in $\mathbb{R}^n\setminus \Omega^s$.

The notions of Steiner symmetral of a set and Steiner rearrangement of a function are clearly related. Actually, if $f:\Omega\rightarrow [0,+\infty)$ is as above, and
\begin{eqnarray}
\mathcal{S}_f=\{(x^{\prime},y,t)\in\mathbb{R}^{n+1}:\;(x^{\prime},y)\in\Omega,\;0<t<f(x^{\prime},y)\},
\end{eqnarray}
the {\it subgraph} of $f$,
 then
\begin{eqnarray}\label{2m}
(\mathcal{S}_f)^s\;{\rm is\;equivalent\;to}\;\mathcal{S}_{f^s}.
\end{eqnarray}

Moreover, for the {\it level set} of $f$ defined by
\begin{eqnarray}
[f]_t=\{(x^{\prime},y)\in\Omega:\;f(x^{\prime},y)>t\},
\end{eqnarray}
we have
\begin{eqnarray}\label{2d}
[f]^s_t\;{\rm is\;equivalent\;to}\;[f^s]_t\;\;{\rm for\;every}\;t>0.
\end{eqnarray}

Let $\mathcal{N}$ be the class of convex functions $\phi:\mathbb{R}\rightarrow [0,\infty)$ such that $\phi(0)=0$ and such that $\phi$ is strictly decreasing on $(-\infty,0]$ or $\phi$ is strictly increasing on $[0,\infty)$.  The subclass of $\mathcal{N}$ consisting of those $\phi\in\mathcal{N}$ that are strictly convex will be denoted by $\mathcal{N}_s$.
Throughout the paper, we always set $\Phi(t):=\max\{\phi(t),\phi(-t)\}$, $t\in[0,\infty)$. It is easily checked that $\Phi(t)$ is a convex function and  strictly increasing on $[0,\infty)$.

We always assume that $\Omega$ is a bounded open subset of $\mathbb{R}^n$. Let $W^{1,\Phi}(\Omega)$ be the first order Orlicz-Sobolev space (see Section \ref{s3} for the precise definition) corresponding to $\Phi$.  Let $W_0^{1,\Phi}(\Omega)$ denote the subspace of $W^{1,\Phi}(\Omega)$ of those functions whose continuation by $0$ outside $\Omega$ belongs to $W^{1,\Phi}(\mathbb{R}^n)$. For $v\in S^{n-1}$ and $f\in  W_0^{1,\Phi}(\Omega)$,  we define
\begin{eqnarray}\label{52}
\|v\|_{f,\phi}=\|\nabla_vf\|_{\phi}=\inf\left\{\lambda>0:\;\frac{1}{|\Omega|}\int_{\Omega}\phi\left(\frac{\nabla_vf}{\lambda}\right)dx\leq 1\right\},
\end{eqnarray}
where $\nabla_vf$ is the directional derivative of $f$ in the direction $v$. The definition immediately provides the extension of $\|\cdot\|_{f,\phi}$ from $S^{n-1}$ to $\mathbb{R}^n$. Now $(\mathbb{R}^n,\|\cdot\|_{f,\phi})$ is the $n$-dimensional Banach space that we shall associate with $f$.
And its unit ball $B_{\phi}(f)=\{x\in\mathbb{R}^n:\|x\|_{f,\phi}\leq 1\}$ is a convex body in $\mathbb{R}^n$. An important fact is that its volume $|B_{\phi}(f)|$ is invariant under affine transformations of the form $x\mapsto Ax+x_0$, with $x_0\in\mathbb{R}^n$ and $A\in SL(n)$. We call the unit ball $B_{\phi}(f)$ the {\it Orlicz-Sobolev ball of $f$}. We call
\begin{eqnarray}\label{1a}
\mathcal {E}_{\phi}(f):=|B_{\phi}(f)|^{-\frac{1}{n}}=\left(\frac{1}{n}\int_{S^{n-1}}\|\nabla_vf\|_{\phi}^{-n}dv\right)^{-\frac{1}{n}}
\end{eqnarray}
the {\it Orlicz-Sobolev affine energy} of $f$.

In this paper, we will prove an affine Orlicz P\'olya-Szeg\" o principle with respect to Steiner symmetrization.

\begin{thm}\label{2.1}
Let $\Omega$ be a bounded open subset of $\mathbb{R}^n$ and let $f$ be a nonnegative function from $W_0^{1,\Phi}(\Omega)$. Then for every Steiner rearrangement $f^s$ of $f$,
\begin{eqnarray}\label{2e}
\mathcal {E}_{\phi}(f^s)\leq \mathcal {E}_{\phi}(f).
\end{eqnarray}
\end{thm}

Using Theorem \ref{2.1} and the convergence property of Steiner symmetrization, we can obtain the affine Orlicz P\'olya-Szeg\" o principle with respect to Schwarz spherical symmetrization.

\begin{thm}\label{2.2}
Let $\Omega$ be a bounded open subset of $\mathbb{R}^n$ and let $f$ be a nonnegative function from $W_0^{1,\Phi}(\Omega)$. Then for every  Schwarz spherical symmetrization $f^{\star}$ of $f$,
\begin{eqnarray}\label{2j}
\mathcal {E}_{\phi}(f^{\star})\leq \mathcal {E}_{\phi}(f).
\end{eqnarray}
\end{thm}

When $\phi(t)=(1-\lambda)(t)_+^p+\lambda(t)_-^p$, where $p>1$, $\lambda\in[0,1]$, $(t)_+:=\max\{t,0\}$ and $(t)_-:=\max\{-t,0\}$, the affine Orlicz P\'olya-Szeg\" o principle becomes the general affine P\'olya-Szeg\" o-type principle established in \cite{Nguyen16}. The symmetric affine P\'olya-Szeg\" o principle \cite{CLYZ09} and the asymmetric affine P\'olya-Szeg\" o principle \cite{HSX12} correspond to the cases of $\lambda=1/2$ and $\lambda=0$, respectively.

In order to state our result about the equality case in (\ref{2e}), we need some assumptions on $f$ and $\Omega$. Consider $f$ first, and set
$$M_f(x^{\prime})=\inf\{t>0:\;\mu_f(x^{\prime},t)=0\}\;\;{\rm for}\;x^{\prime}\in\pi_{n-1}(\Omega).$$
Obviously, $M_f(x^{\prime})$ agrees with ${\rm ess}\sup\{f(x^{\prime},y):\;y\in\Omega_{x^{\prime}}\}$ for $\mathcal{L}^{n-1}$-a.e. $x^{\prime}\in\pi_{n-1}(\Omega)$. Moreover, $M$ is a measurable function in $\pi_{n-1}(\Omega)$ with $M_f(x^{\prime})<\infty$ for $\mathcal{L}^{n-1}$-a.e. $x^{\prime}\in\pi_{n-1}(\Omega)$, owing to (\ref{2g}). We demand that, for $\mathcal{L}^{n-1}$-a.e. $x^{\prime}\in\pi_{n-1}(\Omega)$, $M_f(x^{\prime})>0$ and that the derivative of the restriction $f(x^{\prime},\cdot)$ is $\mathcal{L}^1$-a.e. different from $0$ in the set where $f(x^{\prime},\cdot)<M_f(x^{\prime})$. This is equivalent to the condition
\begin{eqnarray}\label{2i}
\mathcal{L}^n(\{(x^{\prime},y)\in\Omega:\nabla_yf(x^{\prime},y)=0\}&\cap&\{(x^{\prime},y)\in\Omega:\;\;M_f(x^{\prime})=0\nonumber\\
&&\;{\rm or}\;f(x^{\prime},y)<M_f(x^{\prime})\})=0.
\end{eqnarray}
As far as $\Omega$ is concerned, we require that
\begin{eqnarray}\label{2b}
\pi_{n-1}(\Omega)\;{\rm is\;connected},
\end{eqnarray}
and that, in a sense, the reduced boundary $\partial^{\ast}\Omega$ of $\Omega$ is almost nowhere parallel to the $y$-axis inside the open cylinder $\pi_{n-1}(\Omega)\times\mathbb{R}_y$. A precise formulation of the last condition reads
\begin{eqnarray}\label{2h}
&&\Omega\;{\rm has\;locally\;finite\;perimeter\;in}\;\pi_{n-1}(\Omega)\times\mathbb{R}_y\;{\rm and}\nonumber\\
&&\mathcal{H}^{n-1}\left(\{(x^{\prime},y)\in\partial^{\ast}\Omega:v_y^{\Omega}(x^{\prime},y)=0\}\cap(\pi_{n-1}(\Omega)\times\mathbb{R}_y)\right)=0,
\end{eqnarray}
where $\mathcal{H}^k$ stands for $k$-dimensional Hausdorff measure, and $v^{\Omega}_y$ denotes the component along the $y$-axis of the generalized inner normal $v^{\Omega}$ to $\Omega$ (see Section \ref{ss1} for definitions).

We are now ready to state our result about the equality case in (\ref{2e}).

\begin{thm}\label{2.3}
Let $\Omega$ be a bounded open subset of $\mathbb{R}^n$ fulfilling {\rm(\ref{2b})--(\ref{2h})} and let $f$ be a nonnegative function from $W_0^{1,\Phi}(\Omega)$ satisfying {\rm (\ref{2i})} and $\phi\in\mathcal{N}_s$. Then
\begin{eqnarray}\label{2r}
\mathcal {E}_{\phi}(f)=\mathcal {E}_{\phi}(f^s)
\end{eqnarray}
if and only if there exist $A\in SL(n)$ and $x_0\in\mathbb{R}^n$ such that
\begin{eqnarray}\label{1b}
f(x)=f^s(Ax+x_0)\;\;for\;\mathcal{L}^n\text{-}a.e.\;x\in\Omega.
\end{eqnarray}
\end{thm}

By Theorem \ref{2.3} and some delicate analyses,  we can characterize  the case of equality in (\ref{2j}). In particular we establish a Brothers-Ziemer type theorem for the affine Orlicz P\'olya-Szeg\" o principle. We demand
\begin{eqnarray}\label{2y}
&&\Omega\;{\rm is\;a\;set\;of\;finite\;perimeter\;in}\;\mathbb{R}^n
\end{eqnarray}
and
\begin{eqnarray}\label{2z}
\mathcal{L}^n\left(\left\{x\in \Omega:\;\nabla f(x)=0\; {\rm and}\;0\leq f(x)<{\rm ess}\sup f\right\}\right)=0.
\end{eqnarray}

\begin{thm}\label{2.4}
 Let $\Omega$ be a bounded and connected open subset of $\mathbb{R}^n$ fulfilling {\rm (\ref{2y})} and let $f$ be a nonnegative function from $W_0^{1,\Phi}(\Omega)$ satisfying {\rm (\ref{2z})} and $\phi\in\mathcal{N}_s$.  Then
\begin{eqnarray}
\mathcal {E}_{\phi}(f)=\mathcal {E}_{\phi}(f^{\star})
\end{eqnarray}
if and only if there exist $A\in SL(n)$ and $x_0\in\mathbb{R}^n$ such that
\begin{eqnarray}
f(x)=f^{\star}(Ax+x_0)\;\;for\;\mathcal{L}^n\text{-}a.e.\;x\in\Omega.
\end{eqnarray}
\end{thm}

\section{Background and Preliminaries}\label{s3}
\subsection{On the Brunn-Minkowski theory of convex bodies.}
In this section we fix our notation and collect basic facts from convex geometric analysis.
General references for the theory of convex bodies are the books by  Gardner \cite{Ga06}, Gruber \cite{Gruber07}, Schneider \cite{Schneider13}.

 We write $\mathcal{K}^n$ for the set of convex bodies (compact convex subsets) of $\mathbb{R}^n$. We write $\mathcal{K}^n_o$ for the set of convex bodies that contain the origin in their interiors. For $K\in \mathcal{K}^n$, let $h(K;\cdot)=h_K:\mathbb{R}^n\rightarrow\mathbb{R}$ denote the {\it support function} of $K$; i.e.,
$$h(K;x):=\max\{x\cdot z:\;z\in K\}.$$
For $K\in \mathcal {K}_o^n$, its {\it gauge function} $g_K:\mathbb{R}^n\rightarrow[0,\infty)$ is defined by
\begin{eqnarray}\label{41b}
g_K(x):=\|x\|_K=\inf\{\lambda>0:\;x\in\lambda K\}.
\end{eqnarray}

If $K\in\mathcal {K}_o^n$, then the {\it polar body} $K^{\ast}$ is defined by
$$K^{\ast}:=\{x\in\mathbb{R}^n:x\cdot z\leq 1\;{\rm for\;all}\;z\in K\}.$$

If $K\in\mathcal{K}_o^n$, it is well-known that
\begin{eqnarray}\label{41a}
g_K=h_{K^{\ast}}.
\end{eqnarray}

By (\ref{41b}), for $x\in\mathbb{R}^n$ and $K\in\mathcal{K}_o^n$,  it follows immediately that
\begin{eqnarray}\label{41c}
g_K(x)=1\;\;{\rm if\;and\;only\;if}\;\;x\in\partial K.
\end{eqnarray}

For $K,L\in\mathcal{K}^n$, the {\it Hausdorff distance} of $K$ and $L$ is defined by
\begin{eqnarray}
\delta(K,L):=\sup_{u\in S^{n-1}}|h_{K}(u)-h_{L}(u)|.
\end{eqnarray}

When considering the convex body $K\in\mathcal {K}_o^n$ as $K\subset \mathbb{R}^{n-1}\times\mathbb{R}_y$,  the {\it Steiner symmetral}, $K^s$, of $K$ in the direction $e_n$ is given by
\begin{eqnarray}\label{44}
K^s=\left\{\left(x^{\prime},\frac{1}{2}y_1+\frac{1}{2}y_2\right)\in\mathbb{R}^{n-1}\times\mathbb{R}:(x^{\prime},y_1),(x^{\prime},-y_2)\in K\right\}.
\end{eqnarray}

In this paper, we shall make use of the following fact that follows directly from (\ref{44}) and (\ref{41c}).
\begin{lem}\label{2f}
Suppose $K,L\in\mathcal {K}_o^n$ and consider $K,L\subset\mathbb{R}^{n-1}\times\mathbb{R}$. Then
$$K^s\subset L,$$
if and only if
$$\|(x_0^{\prime},\eta_1)\|_K=1=\|(x_0^{\prime},-\eta_2)\|_{K},\;\; with\;\eta_1\neq-\eta_2\;\;\Longrightarrow \;\;\|(x^{\prime},\eta_1/2+\eta_2/2)\|_{L}\leq 1.$$
In addition, if $K^s=L$, then $\|(x_0^{\prime},\eta_1)\|_K=1=\|(x_0^{\prime},-\eta_2)\|_K$ with $\eta_1\neq-\eta_2$ implies that $\|(x_0^{\prime},\eta_1/2+\eta_2/2)\|_L= 1$.
\end{lem}
\subsection{On the theory of Sobolev functions}
In this section, we review some basic definitions and facts about Sobolev functions and functions of bounded variation on $\mathbb{R}^n$.  Good general references for this are Ambrosio, Fusco and Pallara \cite{AFP00}, Cianchi and Fusco \cite{CF06}, Evans and Gariepy \cite{EG92}, Ma${\rm z}^{\prime}$ya \cite{Ma85} and Ziemer \cite{Ziemer89}.

\subsubsection{On Orlicz-Sobolev functions}

  Let $\Omega$ be an open subset of $\mathbb{R}^n$ and let $\phi\in\mathcal{N}$. Let $\Phi(t)=\max\{\phi(t),\phi(-t)\}$, $t\in[0,\infty)$. The Orlicz space $L^{\Phi}(\Omega)$ is defined as
\begin{eqnarray}\label{33a}
L^{\Phi}(\Omega)&:=&\left\{f:\;f \;{\rm is\;a\;Lebesgue\;measurable\;real\;valued\;function\;on}\;\Omega\right.\nonumber\\
&~&\left.{\rm such\;that}\;\int_{\Omega}\Phi\left(\frac{|f(x)|}{\lambda}\right)dx<\infty\;{\rm for\;some\;}\lambda>0\right\}.
\end{eqnarray}
The Luxemburg norm $\|f\|_{L^{\Phi}(\Omega)}$ is defined as
\begin{eqnarray}\label{33}
\|f\|_{L^{\Phi}(\Omega)}:=\inf\left\{\lambda>0:\int_{\Omega}\Phi\left(\frac{|f(x)|}{\lambda}\right)dx\leq 1\right\}.
\end{eqnarray}
The space $L^{\Phi}(\Omega)$, equipped with the norm $\|\cdot\|_{L^{\Phi}(\Omega)}$, is a Banach space. Note that, if $\Phi(s)=s^p$ and $p> 1$, then $L^{\Phi}(\Omega)=L^p(\Omega)$, the usual $L^p$ space, and $\|\cdot\|_{L^{\Phi}(\Omega)}=\|\cdot\|_{L^{p}(\Omega)}$. Usually, we write $\|\cdot\|_{\Phi}$ instead of $\|\cdot\|_{L^{\Phi}(\Omega)}$.

The first order Orlicz-Sobolev space $W^{1,\Phi}(\Omega)$ is defined as
\begin{eqnarray}\label{6b}
W^{1,\Phi}(\Omega)=\{f\in L^{\Phi}(\Omega): f\;{\rm is}\;{\rm weakly}\;{\rm differentiable}\;{\rm and\;}|\nabla f|\in L^{\Phi}(\Omega)\}.
\end{eqnarray}
Here, $\nabla$ denotes the approximate
gradient (see the definition in (\ref{1c})). By $W^{1,\Phi}_{\rm loc}(\Omega)$ we denote the space of those functions which belong to $W^{1,\Phi}(\Omega^{\prime})$ for every open set $\Omega^{\prime}\Subset\Omega$.

The space $W^{1,\Phi}(\Omega)$, equipped with the norm
\begin{eqnarray}
\|f\|_{W^{1,\Phi}(\Omega)}=\|f\|_{\Phi}+\||\nabla f|\|_{\Phi},
\end{eqnarray}
is a Banach space. Clearly, $W^{1,\Phi}(\Omega)=W^{1,p}(\Omega)$, the standard Sobolev space, if $\Phi(s)=s^p$ with $p>1$.

We shall make use of the following trivial fact.
\begin{lem}\label{2q}
If $\phi\in\mathcal{N}$, then for $a,b\in\mathbb{R}$ and $a\neq 0$, the function
\begin{eqnarray}\label{6}
\Psi(t):=\phi(at-b)+\phi(-at-b),\;t>0
\end{eqnarray}
is  increasing. In addition, if $\phi\in \mathcal{N}_s$, then $\Psi(t)$ is strictly increasing.
\end{lem}

Since $\Omega$ is a bounded open set and $\Phi$ is a convex function and strictly increasing on $[0,\infty)$, we have the following easily-established result.
\begin{lem}\label{6g}
If $f\in W_0^{1,\Phi}(\Omega)$, then $f\in W_0^{1,1}(\Omega)$.
\end{lem}

By \cite[Lemma 2.7]{BCFP} and \cite[Theorem 2.1]{CF06}, we have the following lemma.

\begin{lem}\label{3d}
Let $\Omega$ be a bounded open subset of $\mathbb{R}^n$. If $f\in W_0^{1,1}(\Omega)$, then $f^s\in W_0^{1,1}(\Omega^s)$.
\end{lem}

 For Sobolev spaces $W^{1,p}(\mathbb{R}^n)$ ($1\leq p<\infty$), Burchard \cite[Proposition 7.1]{Bu97} proved the following proposition on the approximation of Schwarz spherical symmetrization by Steiner symmetrizations.
 \begin{pro}\label{19f}  {\rm (Convergence of the $W^{1,p}$-norm)}. Let $f$ be a nonnegative function in $W^{1,p}(\mathbb{R}^n)$, $n\geq2$ and $p\geq 1$, that vanishes at infinity. There exists a sequence of successive Steiner symmetrizations $\{f_k\}_{k\geq 0}$ of $f$  so that
 $$f_k\rightharpoonup f^{\star}\;{\rm weakly\;in}\;W^{1,1}(\mathbb{R}^n).$$
 \end{pro}
\begin{rem}\label{19b}
 For Orlicz-Sobolev spaces $W^{1,\Phi}(\mathbb{R}^n)$, there does not exist a result similar to that of Proposition \ref{19f}. Thus we consider the problem in $W^{1,1}(\mathbb{R}^n)$. By Lemma \ref{6g}, $f\in W_0^{1,1}(\Omega)$  for $f\in W_0^{1,\Phi}(\Omega)$.  Thus, for $f\in W_0^{1,\Phi}(\Omega)$, there exists a sequence of successive Steiner symmetrizations $\{f_k\}_{k\geq 0}$ of $f$ so that
$$f_k\rightharpoonup f^{\star}\;{\rm weakly\;in}\;W^{1,1}(\mathbb{R}^n).$$
\end{rem}

\subsubsection{On functions of bounded variation and sets of finite perimeter}\label{ss1}
The space of functions of bounded variation in $\Omega$ is denoted by $BV(\Omega)$. Recall that a function $f\in L^1(\Omega)$ is said to be of bounded variation in $\Omega$ if its distributional gradient $Df$ is a vector-valued Radon measure in $\Omega$ whose total variation $|Df|$ is finite in $\Omega$ (see the precise definitions in \cite[p.196]{EG92}).
The space $BV_{\rm loc}(\Omega)$ is defined accordingly.

Given a measurable set $E$ in $\mathbb{R}^n$ and a
point $x\in\mathbb{R}^n$, the density of $E$ at $x$ is defined by
$$D(E,x)=\lim_{r\rightarrow 0}\frac{\mathcal{L}^n(E\cap B_r(x))}{\mathcal{L}^n(B_r(x))},$$
provided that the limit on the right-hand side exists. Here, $B_r(x)$ denotes the ball,
centered at $x$, having radius $r$. The {\it essential boundary} of $E$ is the Borel set
\begin{eqnarray}
\partial^M E=\mathbb{R}^n\setminus\{x\in\mathbb{R}^n:\;\;D(E,x)=0\;{\rm or}\;D(E,x)=1\}.
\end{eqnarray}
One has
\begin{eqnarray}
\partial^M(E^{\prime}\cup E^{\prime\prime})\cup\partial^M(E^{\prime}\cap E^{\prime\prime})\subset\partial^ME^{\prime}\cup\partial^ME^{\prime\prime}
\end{eqnarray}
for any measurable sets $E^{\prime}$ and $E^{\prime\prime}$ in $\mathbb{R}^n$.

 For any measurable function $f$ in an open set $\Omega\subset\mathbb{R}^n$, the {\it approximate upper} and
{\it lower limit} of $f$ at a point $x$ are defined as
\begin{eqnarray}
f_+(x)=\inf\{t:D(\{f>t\},x)=0\}\;{\rm and}\; f_-(x)=\sup\{t:D(\{f < t\},x)=0\} ,
\end{eqnarray}
respectively. The function $f$ is said to be {\it approximately continuous} at $x$ if $f_+(x)$ and
$f_-(x)$ are equal and finite. The common value of $f_+(x)$ and $f_-(x)$ at a point of
 approximate continuity $x$ is called the {\it approximate limit} of $f$ at $x$ and is denoted
by $\tilde{f}(x)$. By $\mathcal{C}_f$ we denote the Borel set of all points at which $f$ is approximately
continuous. The {\it precise representative} $f^{\ast}$ of $f$ is defined as
\begin{equation}
f^{\ast}(x)=\left\{\begin{aligned}
&\frac{f_+(x)+f_-(x)}{2}&&{\rm if}\; f_+(x)\;{\rm and}\; f_-(x) \;{\rm are\; both\; finite},\\
&0&&{\rm otherwise}.
\end{aligned} \right.
\end{equation}

Clearly, $f^{\ast}\equiv \tilde{f}$ in $\mathcal{C}_f$. A locally integrable function $f$ in $\Omega$ is said to be {\it approximately differentiable} at $x\in\mathcal{C}_f$ if there exists a vector $\nabla f(x)$ in $\mathbb{R}^n$, called the {\it approximate gradient} of $f$ at $x$, such that
\begin{eqnarray}\label{1c}
\lim_{r\rightarrow 0}\frac{1}{\mathcal{L}(B_r(x))}\int_{B_r(x)}\frac{|f(z)-\tilde{f}(x)-\langle\nabla f(x),z-x\rangle|}{r}dz=0.
\end{eqnarray}
The set of all points $x\in\mathcal{C}_f$ where $f$ is approximately differentiable is a Borel set
denoted by $\mathcal{D}_f$. The subset of $\mathcal{D}_f$ where $\nabla f\neq0$ and the subset where $\nabla f=0$ will be
denoted by $\mathcal{D}_f^+$ and $\mathcal{D}^-_f$, respectively. If $f\in BV(\Omega)$, then $\mathcal{L}^n(\Omega\setminus\mathcal{D}_f)=0$. Moreover, denoting by $D^af$ and by $D^sf$ the absolutely continuous part and the singular part, respectively, of $Df$ with respect to $\mathcal{L}^n$, we have that $\nabla f$ agrees $\mathcal{L}^n$-a.e. with the density of $D^af$ with respect to $\mathcal{L}^n$, and that $|D^sf|(\mathcal{D}_f)=0$. Thus, in particular, $W^{1,1}(\Omega)$ can be identified with the subspace of $BV(\Omega)$ of those functions in $BV(\Omega)$ such that $|Df|(B)=0$ for every Borel set $B\subset \Omega$ satisfying $\mathcal{L}^n(B)=0$.

A measurable subset $E$ of $\mathbb{R}^n$ is said to be of {\it finite perimeter} in an open set $\Omega\subset\mathbb{R}^n$ if $D\chi_E$ is a vector-valued Radon measure with finite total variation in $\Omega$, where $\chi_E$ denotes the characteristic function of $E$. The perimeter of $E$ in a Borel subset $B$ of $\Omega$ is defined by
$$P(E;B)=|D\chi_E|(B).$$
When $B=\mathbb{R}^n$, we shall simply write $P(E)$ instead of $P(E;\mathbb{R}^n)$. If $\chi_E\in BV_{\rm loc}(\Omega)$, then we say that $E$ has {\it locally finite perimeter} in $\Omega$.

The following theorem (see \cite[Sect. 4.1.5, Theorem 1]{GMS98}) completely characterizes functions of bounded variation in terms of their subgraphs. Let us remark that a slightly different notion of subgraph is needed here. Given a function $f:\Omega\subset \mathbb{R}^n\rightarrow \mathbb{R}$, we set
$$\mathcal{S}_f^-:=\{(x,y,t)\in\mathbb{R}^{n+1}:(x,y)\in\Omega,\;t<f(x,y)\}.$$

\begin{thm}\label{2s}
Let $\Omega$ be a bounded open subset of $\mathbb{R}^n$ and let $f$ be a nonnegative function from $L^1(\Omega)$. Then $\mathcal{S}^-_f$ is a set of finite perimeter in $\Omega\times\mathbb{R}_t$ if and only if $f\in BV(\Omega)$. Moreover, in this case,
$$P(\mathcal{S}^-_f;B\times\mathbb{R}_t)=\int_{B}\sqrt{1+|\nabla f|^2}+|D^sf|(B)$$
for every Borel set $B\subset \Omega$.
\end{thm}
Let $E$ be a set of locally finite perimeter in an open subset $\Omega$ of $\mathbb{R}^n$ and let $D_i\chi_E$ denote the $i$-th component of the distributional gradient $D\chi_E$. We denote by $v_i^E$, $i=1,\dots,n$, the derivative of the measure $D_i\chi_E$ with respect to $|D\chi_E|$. The {\it reduced boundary}  $\partial^{\ast}E$ of $E$ is the set of all points $x\in\Omega$ such that the vector $v^E(x)=(v_1^E(x),\dots,v_n^E(x))$ exists and satisfies $|v^E(x)|=1$ (see the precise definitions in \cite[p.221]{EG92}). The vector $v^E(x)$ is called the {\it generalized inner normal} to $E$ at $x$.

\begin{thm}\cite[Theorem B]{CF06}\label{2w}
Let $\Omega$ be an open subset of $\mathbb{R}^n$ and let $E^{\prime}$ and $E^{\prime\prime}$ be sets of locally finite perimeter in $\Omega$. Then $v^{E^{\prime}}(x)=\pm v^{E^{\prime\prime}}(x)$ for $\mathcal{H}^{n-1}$-a.e. $x\in\partial^{\ast}E^{\prime}\cap\partial^{\ast}E^{\prime\prime}$.
\end{thm}

If $f\in W^{1,1}(\Omega)$ and $g:\Omega\rightarrow[0,+\infty)$ is a Borel function, then {\it coarea formula} for Sobolev functions can be written as
\begin{eqnarray}\label{2p}
\int_{\Omega}g|\nabla f|dx=\int_{-\infty}^{+\infty}dt\int_{\Omega\cap\partial^{\ast}\{f>t\}}g
d\mathcal{H}^{n-1}=\int_{-\infty}^{+\infty}dt\int_{\{f^{\ast}=t\}}gd\mathcal{H}^{n-1}.
\end{eqnarray}

The following proposition is a special case of the coarea formula for rectifiable sets (see\cite[Theorem 2.93]{AFP00}).

\begin{pro}
Let $\Omega\subset\mathbb{R}^n$ be an open set and let $E$ be a set of finite perimeter in $\Omega$. Let $g:\Omega\rightarrow [0,+\infty]$ be a Borel function. Then
\begin{eqnarray}
\int_{\partial^{\ast}E\cap\Omega}g(x)|v^E_n(x)|d\mathcal{H}^{n-1}(x)=\int_{\pi_{n-1}(\Omega)}dx^{\prime}\int_{(\partial^{\ast}E\cap\Omega)_{x^{\prime}}}g(x^{\prime},y)d\mathcal{H}^0(y).
\end{eqnarray}
\end{pro}

The next theorem links the approximate gradient of a function of bounded variation to the generalized inner normal to its subgraph (see \cite[Sect. 4.1.5, Theorems 4 and 5]{GMS98}).
\begin{thm}\label{T1}
Let $\Omega$ be an open subset of $\mathbb{R}^n$ and let $\nabla_if$ denote the $i$-th component of $\nabla f$. Then for $f\in BV(\Omega)$,
\begin{eqnarray}\label{2c}
v^{\mathcal{S}^-_f}(x,t)=\left(\frac{\nabla_1f(x)}{\sqrt{1+|\nabla f(x)|^2}},\cdots,\frac{\nabla_n f(x)}{\sqrt{1+|\nabla f(x)|^2}},\frac{-1}{\sqrt{1+|\nabla f(x)|^2}}\right)
\end{eqnarray}
for $\mathcal{H}^n$-a.e. $(x,t)\in \partial^{\ast}\mathcal{S}^-_f\cap(\mathcal{D}_f\times \mathbb{R}_t)$ and
$$v^{\mathcal{S}^-_f}_{n+1}(x,t)=0\;\; for\;\mathcal{H}^n\text{-}a.e.\;(x,t)\in\partial^{\ast}\mathcal{S}^-_f\cap[(\Omega\setminus\mathcal{D}_f)\times\mathbb{R}_t].$$
In particular, if $f\in W^{1,1}(\Omega)$, then {\rm (\ref{2c})} holds for $\mathcal{H}^n$-a.e. $(x,t)\in\partial^{\ast}\mathcal{S}_f^-\cap(\Omega\times \mathbb{R}_t)$.
\end{thm}

In what follows, the essential projection of a set $E\subset\mathbb{R}^{n+1}$ onto $\mathbb{R}^{n-1}\times\mathbb{R}_t$ is defined as
$$\pi_{n-1,n+1}(E)^+=\{(x^{\prime},t)\in\mathbb{R}^{n-1}\times\mathbb{R}_t:\;\ell_E(x^{\prime},t)>0\}.$$
The essential projection $\pi_{n-1}(E)^+$ onto $\mathbb{R}^{n-1}$ is defined similarly.

Finally, we give a theorem concerning one-dimensional sections of sets of finite perimeter. The result is due to Vol'pert \cite{Vol67}. In the present form, it can be easily deduced from \cite[Theorem 3.108]{AFP00}.

\begin{thm}\label{2t}
Let $E$ be a set of finite perimeter in $\Omega$. Then, for $\mathcal{L}^{n-1}$-a.e. $x^{\prime}\in\pi_{n-1}(\Omega)$,
\begin{eqnarray}\label{2u}
E_{x^{\prime}}\; has\;finite\;perimeter\;in\;\Omega_{x^{\prime}},
\end{eqnarray}
\begin{eqnarray}
(\partial^{\ast}E\cap \Omega)_{x^{\prime}}=\partial^{\ast}(E_{x^{\prime}})\cap \Omega_{x^{\prime}},
\end{eqnarray}
\begin{eqnarray}
v^E_n(x^{\prime},y)\neq 0\; for\;every\;y\; such\;that\;(x^{\prime},y)\in\partial^{\ast}E\cap \Omega,
\end{eqnarray}
\begin{equation}\label{2v}
\left\{ \begin{aligned}
\lim_{z\rightarrow y^+}\chi^{\ast}_E(x^{\prime},z)=1,\;\;\lim_{z\rightarrow y^-}\chi^{\ast}_E(x^{\prime},z)=0\;\;{\rm if}\;v_n^E(x^{\prime},y)>0,\\
\lim_{z\rightarrow y^+}\chi^{\ast}_E(x^{\prime},z)=0,\;\;\lim_{z\rightarrow y^-}\chi^{\ast}_E(x^{\prime},z)=1\;\;{\rm if}\;v_n^E(x^{\prime},y)<0.
\end{aligned} \right.
\end{equation}
In particular, there exists a Borel set $\Omega_E\subset\pi_{n-1}(E)^+\cap\pi_{n-1}(\Omega)$ satisfying $\mathcal{L}^{n-1}(\pi_{n-1}(E)^+\cap\pi_{n-1}(\Omega)\setminus \Omega_E)=0$ and such that {\rm (\ref{2u})-(\ref{2v})} hold for every $x^{\prime}\in \Omega_E$.
\end{thm}

\subsection{On the Orlicz-Sobolev balls}
Let $\Omega$ be a bounded open subset of $\mathbb{R}^n$. By (\ref{6b}), if $f\in W_0^{1,\Phi}(\Omega)$, then $\left|\nabla f\right|\in L^{\Phi}(\Omega)$. Thus by (\ref{33a}), there exists some $\lambda>0$ such that
$$\int_{\Omega}\Phi\left(\frac{|\nabla f(x)|}{\lambda}\right)dx<\infty.$$
Since $\phi\in\mathcal{N}$ and $\Phi(t)=\max\{\phi(t),\phi(-t)\}$, $t\in [0,\infty)$, there exists some $\lambda>0$ such that for any $u\in S^{n-1}$,
\begin{eqnarray}
\int_{\Omega}\phi\left(\frac{u\cdot\nabla f(x)}{\lambda}\right)dx\leq \int_{\Omega}\Phi\left(\frac{|\nabla f(x)|}{\lambda}\right)dx<\infty.
\end{eqnarray}
 Thus, we can define the Orlicz-Sobolev ball $B_{\phi}(f)$ of $f\in W_0^{1,\Phi}(\Omega)$ as the unit ball of the $n$-dimensional Banach space whose
norm is given by
\begin{eqnarray}\label{2ff}
\|z\|_{f,\phi}:=\inf\left\{\lambda> 0:\;\frac{1}{|\Omega|}\int_{\Omega}\phi\left(\frac{z\cdot\nabla f(x)}{\lambda}\right)dx\leq 1\right\},\;z\in\mathbb{R}^n.
\end{eqnarray}

And the volume of the Orlicz-Sobolev ball is given by
\begin{eqnarray}
|B_{\phi}(f)|=\frac{1}{n}\int_{S^{n-1}}\|v\|_{f,\phi}^{-n}dv,
\end{eqnarray}
where $dv$ denotes the spherical Lebesgue measure.

Since $f\in W_0^{1,\Phi}(\Omega)$, it is impossible that there exists some $u_0\in S^{n-1}$ such that $\nabla f(x)\cdot u_0\geq0$  for almost all $x\in \Omega$. Since $\phi$ is strictly increasing on $[0,\infty)$ or strictly decreasing on $(-\infty,0]$, it follows that for $z\neq0$ the function
$$\lambda\mapsto\frac{1}{|\Omega|}\int_{\Omega}\phi\left(\frac{z\cdot\nabla f(x)}{\lambda}\right)dx$$
is strictly decreasing in $(0,\infty)$. Thus, we have the following lemma.
\begin{lem}\label{2k}
Let $f\in W_0^{1,\Phi}(\Omega)$ and $z_0\in\mathbb{R}^n\backslash \{0\}$. Then
$$\frac{1}{|\Omega|}\int_{\Omega}\phi\left(\frac{z_0\cdot\nabla f(x)}{\lambda_0}\right)dx=1$$
if and only if
$$\|z_0\|_{f,\phi}=\lambda_0.$$
\end{lem}

 The following Lemma \ref{2dd}, Lemma \ref{2a} and Lemma \ref{1d}  demonstrate the affine invariance of $\mathcal{E}_{\phi}(f)$, the non-negativity and boundedness of $\|\cdot\|_{f,\phi}$, respectively. Since their proofs are the same as the proofs of Lemma 4.2--Lemma 4.4 in \cite{LYJ17}, we omit their proofs.

\begin{lem}\label{2dd} (\cite[Lemma 4.2]{LYJ17})
If  $f\in W_0^{1,\Phi}(\Omega)$, then
$\mathcal{E}_{\phi}(f)$
is invariant under $SL(n)$ transformations and translations.
\end{lem}

\begin{lem}\label{2a} (\cite[Lemma 4.3]{LYJ17})
If $f\in W_0^{1,\Phi}(\Omega)$, then $\|\cdot\|_{f,\phi}$ defines a norm on the Banach space $(\mathbb{R}^n,\|\cdot\|_{f,\phi})$. In particular, $\|v\|_{f,\phi}>0$ for any $v\in S^{n-1}$.
\end{lem}

\begin{lem}\label{1d} {(\cite[Lemma 4.4]{LYJ17})}
If $f\in W_0^{1,\Phi}(\Omega)$ and
\begin{eqnarray}\label{1e}
c_{\phi}=\max\left\{c>0:\max\{\phi(c),\phi(-c)\}\leq 1\right\},
\end{eqnarray}
then for any $v\in S^{n-1}$, we have
\begin{eqnarray}
\frac{\int_{\Omega}f(x)dx}{c_{\phi}|\Omega|{\rm diam}(\Omega)}\leq \|v\|_{f,\phi}\leq \frac{\sup\{|\nabla f(x)|:x\in \Omega\}}{c_{\phi}},
\end{eqnarray}
where ${\rm diam}(\Omega):=\sup\{|x-y|:\;x,y\in\Omega\}$ denotes the diameter of $\Omega$.
\end{lem}

The following lemma shows that the Orlicz-Sobolev  ball operator $B_{\phi}: W_0^{1,\Phi}(\Omega)\rightarrow \mathcal{K}_o^n$ is in some sense weakly continuous.
\begin{lem}\label{2x}
Let $f_i\in W_0^{1,\Phi}(\Omega_i)$, $i=0,1,2,\dots$. If
\begin{eqnarray}\label{1h}
f_i\rightharpoonup f_0,\; weakly\;in\;W^{1,1}(\mathbb{R}^n),
\end{eqnarray}
then there exists a subsequence of  $\{B_{\phi}(f_i)\}_{i=1}^{\infty}$, denoted by   $\{B_{\phi}(f_i)\}_{i=1}^{\infty}$ as well, and a convex body $K_0$ such that $o\in K_0$,
\begin{eqnarray}\label{1j}
\lim_{i\rightarrow \infty}\delta(B_{\phi}(f_i),K_0)=0
\end{eqnarray}
and
\begin{eqnarray}\label{1k}
K_0\subset B_{\phi}(f_0).
\end{eqnarray}
\end{lem}

\begin{proof}
For $u_0\in S^{n-1}$, let
\begin{eqnarray}\label{1f}
\|u_0\|_{f_i,\phi}=\lambda_i,
\end{eqnarray}
and note that Lemma \ref{1d} gives
\begin{eqnarray}\label{1g}
0<\frac{\int_{\Omega_i}f_i(x)dx}{c_{\phi}|\Omega_i|{\rm diam}(\Omega)}\leq\lambda_i.
\end{eqnarray}
Moreover, by (\ref{1h}), we have
\begin{eqnarray}\label{1i}
\lim_{i\rightarrow\infty}\frac{\int_{\Omega_i}f_i(x)dx}{c_{\phi}|\Omega_i|{\rm diam}(\Omega)}=\frac{\int_{\Omega_0}f_0(x)dx}{c_{\phi}|\Omega_0|{\rm diam}(\Omega_0)}>0,
\end{eqnarray}
which implies that there exists a real number $m>0$ such that $\|u\|_{f_i,\phi}>m$ for any $u\in S^{n-1}$ and any positive integer $i$. Thus the radial functions  $\rho(B_{\phi}(f_i),u)<\frac{1}{m}$ for any $u\in S^{n-1}$ and any positive integer $i$. By the Blaschke selection theorem
(see \cite[Theorem 1.8.7]{Schneider13}), there exists a subsequence of $\{B_{\phi}(f_i)\}_{i=1}^{\infty}$,  denoted by   $\{B_{\phi}(f_i)\}_{i=1}^{\infty}$ as well, that converges to a convex body $K_0\in\mathcal{K}^n$.
By Lemma \ref{2a}, $\|u\|_{f_i,\phi}>0$ for any $u\in S^{n-1}$ and any positive integer $i$. Thus $o\in K_0$.

Let $\|u_0\|_{K_0}=\lambda_{\ast}$. By (\ref{1j}) and (\ref{1f}), we have
\begin{eqnarray}
\lim_{i\rightarrow \infty}\lambda_i=\lambda_{\ast}.
\end{eqnarray}

Let $\bar{f_i}$ denote the continuation of $f$ by $0$ outside $\Omega_i$ and $\tilde{f_i}=\bar{f_i}/\lambda_i$. Since $\lambda_i \rightarrow\lambda_{\ast}$ and $\bar{f}_i\rightharpoonup\bar{f}_0$  weakly in  $W^{1,1}(\mathbb{R}^n)$, we have
\begin{eqnarray}\label{1l}
\tilde{f}_i\rightharpoonup \bar{f}_0/\lambda_{\ast}\;\;{\rm weakly\;in}\;\;W^{1,1}(\mathbb{R}^n).
\end{eqnarray}

The fact that $\|u_0\|_{f_i,\phi}=\lambda_i$, together with Lemma \ref{2k}, shows that
\begin{eqnarray}\label{1m}
\frac{1}{|\Omega_i|}\int_{\mathbb{R}^n}\phi\left(u_0\cdot \nabla\tilde{f}_i(x)\right)dx=1\;{\rm for\;all}\;i.
\end{eqnarray}

Since $\phi$ is a convex function, by \cite[Theorem 1 in P.19]{Evans90}, the convex gradient integral
$$\frac{1}{|\Omega|}\int_{\mathbb{R}^n}\phi\left(u_0\cdot \nabla\bar{f}(x)\right)dx$$
is lower semicontinuous with respect to weak convergence in $W^{1,1}(\mathbb{R}^n)$. By (\ref{1l}) and (\ref{1m}), we have
$$\frac{1}{|\Omega_0|}\int_{\mathbb{R}^n}\phi\left(\frac{u_0\cdot \nabla\bar{f}_0(x)}{\lambda_{\ast}}\right)dx\leq\lim_{i\rightarrow \infty}\frac{1}{|\Omega_i|}\int_{\mathbb{R}^n}\phi\left(u_0\cdot \nabla\tilde{f}_i(x)\right)dx= 1.$$

This, together with the definition (\ref{2ff}) yields
\begin{eqnarray}\label{1n}
\|u_0\|_{f_0,\phi}\leq\lambda_{\ast}=\|u_0\|_{K_0}.
\end{eqnarray}
By (\ref{1n}) and the arbitrariness of $u_0\in S^{n-1}$,  we have $K_0\subset B_{\phi}(f_0)$.
\end{proof}

\section{Proof of Theorem \ref{2.1}\;-- \ref{2.3}}\label{s4}
\begin{lem}\cite[Proposition 2.3.]{CF06}\label{3f}
Let $\Omega$ be a bounded open subset of $\mathbb{R}^n$. Let $f$ be a nonnegative function from $W_0^{1,1}(\Omega)$. Then for $\mathcal{L}^{n-1}$-a.e. $x^{\prime}\in\pi_{n-1}(\Omega)$,
\begin{eqnarray}
&&\mathcal{L}^1\left(\{y:\nabla_yf(x^{\prime},y)=0,\;t<f(x^{\prime},y)<M_f(x^{\prime})\}\right)\nonumber\\
&=&\mathcal{L}^1\left(\{y:\nabla_yf^s(x^{\prime},y)=0,\;t<f^s(x^{\prime},y)<M_f(x^{\prime})\}\right)
\end{eqnarray}
for every $t\in (0,M_f(x^{\prime}))$.
\end{lem}

\begin{lem}\cite[Lemma 4.1]{CF06}\label{4e}
Let $\Omega$ be a bounded open subset of $\mathbb{R}^n$, and let $f$ be a nonnegative function from $f\in W_0^{1,1}(\Omega)$. Then $\mu_f\in BV(\pi_{n-1}(\Omega)\times\mathbb{R}_t^+)$, and, for $\mathcal{L}^{n-1}$-a.e. $x^{\prime}\in \pi_{n-1}(\mathcal{S}_f)^+$,
\begin{eqnarray}
\nabla_t\mu_f(x^{\prime},t)=-\int_{\partial^{\ast}\{y:f(x^{\prime},y)>t\}}\frac{1}{|\nabla_yf|}d\mathcal{H}^0,
\end{eqnarray}
\begin{eqnarray}
\nabla_i\mu_f(x^{\prime},t)=\int_{\partial^{\ast}\{y:f(x^{\prime},y)>t\}}\frac{\nabla_i f}{|\nabla_yf|}d\mathcal{H}^{0}\;\;\;i=1,\dots,n-1,
\end{eqnarray}
for $\mathcal{L}^1$-a.e. $t\in(0,M_f(x^{\prime}))$.
\end{lem}

\begin{cor}
Let $\Omega$ be a bounded open subset of $\mathbb{R}^n$, and let $f$ be a nonnegative function from $f\in W_0^{1,1}(\Omega)$. Then for $\mathcal{L}^{n-1}$-a.e. $x^{\prime}\in\pi_{n-1}(\mathcal{S}_f)^+$ and $\mathcal{L}^1$-a.e. $t\in (0, M_f(x^{\prime}))$,
\begin{eqnarray}\label{3h}
\nabla_t\mu_f(x^{\prime},t)=-\frac{2}{|\nabla_yf^s(x^{\prime},y_1)|}\;\;and\;\;
\nabla_i\mu_f(x^{\prime},t)=2\frac{\nabla_if^s(x^{\prime},y_1)}{|\nabla_yf^s(x^{\prime},y_1)|},
\end{eqnarray}
where $y_1\in\partial^{\ast}\{y:f^s(x^{\prime},y)>t\}$.
\end{cor}
\begin{proof}
By Lemma \ref{3d}, the function $f^s\in W_0^{1,1}(\Omega^s)$. Moreover, by (\ref{2m}), $\pi_{n-1}(\mathcal{S}_f)^+$ is equivalent to $\pi_{n-1}(\mathcal{S}_{f^s})^+$. Since $\mu_{f^s}=\mu_f$, an application of Lemma \ref{4e} to $f^s$ yields (\ref{3h}).
\end{proof}

\begin{lem}\label{L7}
Let $\Omega$ be a bounded open subset of $\mathbb{R}^n$. If $f$ is a nonnegative function from $W_0^{1,\Phi}(\Omega)$ and
\begin{eqnarray}
n(x^{\prime},t):=\mathcal{H}^0\left(\partial^{\ast}\{y\in\Omega_{x^{\prime}}:\;f(x^{\prime},y)>t\}\right),
\end{eqnarray}
then for $\mathcal{L}^{n-1}$-a.e. $x^{\prime}\in\pi_{n-1}(\Omega)$ such that $M_f(x^{\prime})>0$, $n(x^{\prime},t)$ is an even number for $\mathcal{L}^1$-a.e. $t\in (0,M_f(x^{\prime}))$.
\end{lem}
\begin{proof}
For $f\in W_0^{1,\Phi}(\Omega)$, by Lemma \ref{6g}, $f\in W_0^{1,1}(\Omega)$. By \cite[Theorem 2.1.4]{Ziemer89}, $f$
has a representative  $\bar{f}$ that is absolutely continuous on almost all line segments in $\Omega$ parallel
to the coordinate axes. Since $\bar{f}$ is absolutely continuous, $\{y\in\Omega_{x^{\prime}}:\;\bar{f}(x^{\prime},y)>t\}$ is an open set. Moreover, since $f$ vanishes on the boundary of $\Omega$,
\begin{eqnarray}\label{6e}
\{y\in\Omega_{x^{\prime}}:\;\bar{f}(x^{\prime},y)>t\}=\bigcup_{k=1}^{\infty}(a_k,b_k)
\end{eqnarray}
for $\mathcal{L}^{n-1}$-a.e. $x^{\prime}\in\pi_{n-1}(\Omega)$ such that $M_f(x^{\prime})>0$ and $\mathcal{L}^1$-a.e. $t\in (0,M_f(x^{\prime}))$, where $a_k$ and $b_k$ are the points on the reduced boundary of $\{y\in\Omega_{x^{\prime}}:\;\bar{f}(x^{\prime},y)>t\}$. For the same $x^{\prime}$ and $t$, since
$\{y\in\Omega_{x^{\prime}}:\;\bar{f}(x^{\prime},y)>t\}$  is equivalent to $\{y\in\Omega_{x^{\prime}}:\; f(x^{\prime},y)>t\}$,
\begin{eqnarray}\label{6f}
\partial^{\ast}\{y\in\Omega_{x^{\prime}}:\;f(x^{\prime},y)>t\}=
\partial^{\ast}\{y\in\Omega_{x^{\prime}}:\;\bar{f}(x^{\prime},y)>t\}.
\end{eqnarray}
Moreover, since $\Omega$ be a bounded open subset of $\mathbb{R}^n$ and $f\in W^{1,1}(\Omega)$, by Theorem \ref{2s}, $\mathcal{S}_f^-$ is a set of finite perimeter in $\Omega\times\mathbb{R}_t$.
  Since $$\left(\mathcal{S}^-_f\right)_{x^{\prime},t}\;{\rm is\;equivalent\;to}\;\{y\in\Omega_{x^{\prime}}:\;f(x^{\prime},y)>t\}$$ for $\mathcal{L}^{n-1}$-a.e. $x^{\prime}\in\pi_{n-1}(\Omega)$ such that $M_f(x^{\prime})>0$ and $\mathcal{L}^1$-a.e. $t\in (0,M_f(x^{\prime}))$ and by (\ref{2u}) in Theorem \ref{2t}, $n(x^{\prime},t)$ is finite. By (\ref{6e}) and (\ref{6f}), $n(x^{\prime},t)$ is an even number.
\end{proof}

\begin{pro}\label{4f}
Let $\Omega$ be a bounded open subset of $\mathbb{R}^n$. If $f$ is a nonnegative function from $W_0^{1,\Phi}(\Omega)$, then
\begin{eqnarray}\label{3e}
(B_{\phi}(f))^s\subset B_{\phi}(f^s).
\end{eqnarray}
\end{pro}

\begin{proof}
Let $(x_0^{\prime},\eta_1),(x_0^{\prime},-\eta_2)\in \mathbb{R}^{n-1}\times \mathbb{R}_y$ and
$$\|(x_0^{\prime},\eta_1)\|_{f,\phi}=1\;\;{\rm and}\;\;\|(x_0^{\prime},-\eta_2)\|_{f,\phi}=1,$$
with $\eta_1\neq-\eta_2$.
By Lemma \ref{2k}, this means that
\begin{eqnarray}\label{3t}
\frac{1}{|\Omega|}\int_{\Omega}\phi\left((x_0^{\prime},\eta_1)\cdot \nabla f(x)\right)dx=1
\end{eqnarray}
and
\begin{eqnarray}\label{3u}
\frac{1}{|\Omega|}\int_{\Omega}\phi\left((x_0^{\prime},-\eta_2)\cdot \nabla f(x)\right)dx=1.
\end{eqnarray}
By Lemma \ref{2f}, the desired inclusion (\ref{3e}) will be established if we can show that
\begin{eqnarray}\label{3v}
\|(x_0^{\prime},\eta_1/2+\eta_2/2)\|_{f^s,\phi}\leq 1.
\end{eqnarray}

{\it Step 1:} We assume here that $f$ is nonnegative function from $W_0^{1,1}(\Omega)$ such that
\begin{eqnarray}\label{5t}
\mathcal{L}^1\left(\{y:\;\nabla_yf(x^{\prime},y)=0\}
\cap\{y:\;0<f(x^{\prime},y)<M_f(x^{\prime})\}\right)=0
\end{eqnarray}
for $\mathcal{L}^{n-1}$-a.e. $x^{\prime}\in\pi_{n-1}(\Omega)$ such that $M_f(x^{\prime})>0$. By Lemma \ref{3f}, (\ref{5t}) is fulfilled with $f$ replaced by $f^s$ as well.  By \cite[Theorem E]{CF06},
\begin{eqnarray}\label{2n}
\frac{df^s(x^{\prime},y)}{dy}=\nabla_y f^s(x^{\prime},y)\;\;{\rm for}\;\mathcal{L}^1{\text-}a.e.\;y\in\Omega_{x^{\prime}}^s.
\end{eqnarray}
Hence, by (\ref{2n}) and the coarea formula (\ref{2p}), for $\mathcal{L}^{n-1}$-a.e. $x^{\prime}\in\pi_{n-1}(\Omega)$ such that $M_f(x^{\prime})>0$, we have
\begin{eqnarray}\label{3w}
&&\int_{\{y:f^s(x^{\prime},y)>0\}}\phi\left(\left(x_0^{\prime},
\frac{\eta_1+\eta_2}{2}\right)\cdot\nabla f^s(x^{\prime},y)\right)dy\nonumber\\
&=&\int_{0}^{M_f(x^{\prime})}dt\int_{\partial^{\ast}\{y:f^s(x^{\prime},y)>t\}}
\frac{1}{|\nabla_yf^s|}\phi\left(\left(x_0^{\prime},\frac{\eta_1+\eta_2}{2}\right)\cdot\nabla f^s(x^{\prime},y)\right)d\mathcal{H}^0.
\end{eqnarray}

Moreover, by (\ref{2m}), for $\mathcal{L}^{n-1}$-a.e. $x^{\prime}\in\pi_{n-1}(\Omega)$ and  $\mathcal{L}^{1}$-a.e. $t\in (0,M_f(x^{\prime}))$, there exist two real numbers $y_1(x^{\prime},t)$ and $y_2(x^{\prime},t)$ such that
\begin{eqnarray}\label{3j}
\{y:\;f^s(x^{\prime},y)>t\}\;{\rm is \;equivalent\;to}\;(y_1(x^{\prime},t),y_2(x^{\prime},t))
\end{eqnarray}
and
\begin{eqnarray}\label{3k}
\nabla_yf^s(x^{\prime},y_1(x^{\prime},t))>0\;{\rm and}\;\nabla_yf^s(x^{\prime},y_2(x^{\prime},t))<0.
\end{eqnarray}
Thus, Eqs. (\ref{3h}), (\ref{3j}) and (\ref{3k}) ensure that, for $\mathcal{L}^{n-1}$-a.e. $x^{\prime}\in\pi_{n-1}(\Omega)$ such that $M_f(x^{\prime})>0$,
\begin{eqnarray}\label{3o}
&&\int_{\partial^{\ast}\{y:f^s(x^{\prime},y)>t\}}\frac{1}{|\nabla_yf^s|}
\phi\left(x_0^{\prime}\cdot(\nabla_1f^s,\dots,\nabla_{n-1}f^s)
+\frac{\eta_1+\eta_1}{2}\cdot\nabla_y f^s\right)d\mathcal{H}^0\nonumber\\
&=&\left.\frac{1}{|\nabla_yf^s|}\phi\left(x_0^{\prime}\cdot
(\nabla_1f^s,\dots,\nabla_{n-1}f^s)+\frac{\eta_1+\eta_2}{2}\cdot\nabla_y f^s\right)\right|_{(x^{\prime},y_1(x^{\prime},t))}\nonumber\\
&&\;+\left.\frac{1}{|\nabla_yf^s|}\phi\left(x_0^{\prime}
\cdot(\nabla_1f^s,\dots,\nabla_{n-1}f^s)+\frac{\eta_1+\eta_2}{2}\cdot\nabla_y f^s\right)\right|_{(x^{\prime},y_2(x^{\prime},t))}\nonumber\\
&=&-\frac{1}{2}\nabla_t\mu_f(x^{\prime},t)
\phi\left(x_0^{\prime}\cdot\left(\frac{\nabla_1\mu_f(x^{\prime},t)}
{-\nabla_t\mu_f(x^{\prime},t)},\cdots,
\frac{\nabla_{n-1}\mu_f(x^{\prime},t)}{-\nabla_t\mu_f(x^{\prime},t)}\right)
+\frac{\eta_1+\eta_2}{2}\cdot\frac{2}{-\nabla_t\mu_f(x^{\prime},t)}\right)\nonumber\\
&&\;-\frac{1}{2}\nabla_t\mu_f(x^{\prime},t)\phi\left(x_0^{\prime}
\cdot\left(\frac{\nabla_1\mu_f(x^{\prime},t)}{-\nabla_t\mu_f(x^{\prime},t)},
\cdots,
\frac{\nabla_{n-1}\mu_f(x^{\prime},t)}{-\nabla_t\mu_f(x^{\prime},t)}\right)
+\frac{\eta_1+\eta_2}{2}\cdot\frac{2}{\nabla_t\mu_f(x^{\prime},t)}\right)\nonumber\\
&=&-\frac{1}{2}\nabla_t\mu_f(x^{\prime},t)\phi\left(\frac{x_0^{\prime}
\cdot\left(\nabla_1\mu_f(x^{\prime},t),\cdots, \nabla_{n-1}\mu_f(x^{\prime},t)\right)+(\eta_1+\eta_2)}{-\nabla_t\mu_f(x^{\prime},t)}
\right)\nonumber\\
&&\;-\frac{1}{2}\nabla_t\mu_f(x^{\prime},t)\phi\left(\frac{x_0^{\prime}\cdot
\left(\nabla_1\mu_f(x^{\prime},t),\cdots, \nabla_{n-1}\mu_f(x^{\prime},t)\right)-(\eta_1+\eta_2)}{-\nabla_t\mu_f(x^{\prime},t)}
\right).
\end{eqnarray}

Let
\begin{eqnarray}\label{3m}
n(x^{\prime},t):=\mathcal{H}^0(\partial^{\ast}\{y:\;f(x^{\prime},y)>t\}).
\end{eqnarray}
By Lemma \ref{L7}, $n(x^{\prime},t)$ is an even number for $\mathcal{L}^{n-1}$-a.e. $x^{\prime}\in\pi_{n-1}(\Omega)$ and $\mathcal{L}^1$-a.e. $t\in(0,M_f(x^{\prime}))$. Let $k(x^{\prime},t):=\frac{1}{2}n(x^{\prime},t)$, then $k(x^{\prime},t)$ is an integer and $k(x^{\prime},t)\geq1$.

Let
\begin{eqnarray}
\partial_l^{\ast}\{y:\;f(x^{\prime},y)>t\}\;{\rm be\;the\;subset\;of}\;\partial^{\ast}\{y:\;f(x^{\prime},y)>t\}\;{\rm satisfying}\;v_y^{\mathcal S_f}(x^{\prime},y,t)>0\nonumber
\end{eqnarray}
and
\begin{eqnarray}
\partial_r^{\ast}\{y:\;f(x^{\prime},y)>t\}\;{\rm be\;the\;subset\;of}\;\partial^{\ast}\{y:\;f(x^{\prime},y)>t\}\;{\rm satisfying}\;v_y^{\mathcal S_f}(x^{\prime},y,t)<0.\nonumber
\end{eqnarray}

By $k(x^{\prime},t)\geq 1$, Lemma \ref{2q} and Lemma \ref{4e}, for $\mathcal{L}^{n-1}$-a.e. $x^{\prime}\in\pi_{n-1}(\Omega)$ such that $M_f(x^{\prime})>0$, we have the last expression of (\ref{3o})
\begin{eqnarray}\label{3p}
&\leq&-\frac{1}{2}\nabla_t\mu_f(x^{\prime},t)\phi\left(\frac{x_0^{\prime}
\cdot\left(\nabla_1\mu_f(x^{\prime},t),\cdots, \nabla_{n-1}\mu_f(x^{\prime},t)\right)+k(x^{\prime},t)(\eta_1+\eta_2)}{-\nabla_t\mu_f(x^{\prime},t)}\right)\nonumber\\
&&\;-\frac{1}{2}\nabla_t\mu_f(x^{\prime},t)\phi\left(\frac{x_0^{\prime}\cdot\left(\nabla_1\mu_f(x^{\prime},t),
\cdots, \nabla_{n-1}\mu_f(x^{\prime},t)\right)-k(x^{\prime},t)(\eta_1+\eta_2)}{-\nabla_t\mu_f(x^{\prime},t)}\right)\nonumber\\
&=&\frac{1}{2}\int_{\partial^{\ast}\{\dots\}}\frac{d\mathcal{H}^0}{|\nabla_yf|}
\phi\left(\frac{x_0^{\prime}\cdot\left(\int_{\partial^{\ast}\{\dots\}}\frac{\nabla_{1}f}{|\nabla_yf|}d\mathcal{H}^0,\cdots, \int_{\partial^{\ast}\{\dots\}}\frac{\nabla_{n-1}f}{|\nabla_yf|}d\mathcal{H}^0\right)+k(x^{\prime},t)(\eta_1+\eta_2)}{\int_{\partial^{\ast}\{\dots\}}\frac{d\mathcal{H}^0}{|\nabla_yf|}}\right)\nonumber\\
&&+\frac{1}{2}\int_{\partial^{\ast}\{\dots\}}\frac{d\mathcal{H}^0}{|\nabla_yf|}\phi\left(\frac{x_0^{\prime}\cdot\left(\int_{\partial^{\ast}\{\dots\}}\frac{\nabla_{1}f}{|\nabla_yf|}d\mathcal{H}^0,\cdots, \int_{\partial^{\ast}\{\dots\}}\frac{\nabla_{n-1}f}{|\nabla_yf|}d\mathcal{H}^0\right)-k(x^{\prime},t)(\eta_1+\eta_2)}{\int_{\partial^{\ast}\{\dots\}}\frac{d\mathcal{H}^0}{|\nabla_yf|}}\right)\nonumber\\
&=&\left[\phi\left(\frac{\int_{\partial_l^{\ast}\{\cdots\}}\frac{(x_0^{\prime},\eta_1)\cdot(\nabla_1f,\cdots,\nabla_{n-1}f,|\nabla_yf|)}{|\nabla_yf|}d\mathcal{H}^0+\int_{\partial_r^{\ast}
\{\cdots\}}\frac{(x_0^{\prime},-\eta_2)\cdot(\nabla_1f,\cdots,\nabla_{n-1}f,-|\nabla_yf|)}{|\nabla_yf|}d\mathcal{H}^0}{\int_{\partial^{\ast}\{\dots\}}\frac{d\mathcal{H}^0}{|\nabla_yf|}}\right)\right.\nonumber\\
&&+\left.\phi\left(\frac{\int_{\partial_l^{\ast}\{\cdots\}}\frac{(x_0^{\prime},-\eta_2)\cdot(\nabla_1f,\cdots,\nabla_{n-1}f,|\nabla_yf|)}{|\nabla_yf|}d\mathcal{H}^0+\int_{\partial_r^{\ast}
\{\cdots\}}\frac{(x_0^{\prime},\eta_1)\cdot(\nabla_1f,\cdots,\nabla_{n-1}f,-|\nabla_yf|)}{|\nabla_yf|}d\mathcal{H}^0}{\int_{\partial^{\ast}\{\dots\}}\frac{d\mathcal{H}^0}{|\nabla_yf|}}\right)\right]\nonumber\\
&&\cdot\frac{1}{2}\int_{\partial^{\ast}\{\dots\}}\frac{d\mathcal{H}^0}{|\nabla_yf|}
\end{eqnarray}
for $\mathcal{L}^1$-a.e. $t\in (0,M_f(x^{\prime}))$. Here $\partial^{\ast}\{\dots\}$ is a shorthand for $\partial^{\ast}\{y:f(x^{\prime},y)>t\}$. Since $\phi$ is a convex function,  Jensen's inequality ensures that the last expression
\begin{eqnarray}\label{3q}
&\leq& \frac{1}{2}\int_{\partial_l^{\ast}\{y:f(x^{\prime},y)>t\}}\frac{1}{|\nabla_yf|}\phi\left((x_0^{\prime},\eta_1)\cdot\left(\nabla_1f,\cdots,\nabla_{n-1}f,|\nabla_yf|\right)\right)d\mathcal{H}^0\nonumber\\
&&\;+\frac{1}{2}\int_{\partial_r^{\ast}\{y:f(x^{\prime},y)>t\}}\frac{1}{|\nabla_yf|}\phi\left((x_0^{\prime},-\eta_2)\cdot\left(\nabla_1f,\cdots,\nabla_{n-1}f,-|\nabla_yf|\right)\right)d\mathcal{H}^0\nonumber\\
&&\;+\frac{1}{2}\int_{\partial_l^{\ast}\{y:f(x^{\prime},y)>t\}}\frac{1}{|\nabla_yf|}\phi\left((x_0^{\prime},-\eta_2)\cdot\left(\nabla_1f,\cdots,\nabla_{n-1}f,|\nabla_yf|\right)\right)d\mathcal{H}^0\nonumber\\
&&\;+\frac{1}{2}\int_{\partial_r^{\ast}\{y:f(x^{\prime},y)>t\}}\frac{1}{|\nabla_yf|}\phi\left((x_0^{\prime},\eta_1)\cdot\left(\nabla_1f,\cdots,\nabla_{n-1}f,-|\nabla_yf|\right)\right)d\mathcal{H}^0\nonumber\\
&=&\frac{1}{2}\int_{\partial^{\ast}\{y:f(x^{\prime},y)>t\}}\frac{1}{|\nabla_yf|}\phi\left((x_0^{\prime},\eta_1)\cdot\left(\nabla_1f,\cdots,\nabla_{n-1}f,\nabla_yf\right)\right)d\mathcal{H}^0\nonumber\\
&&\;+\frac{1}{2}\int_{\partial^{\ast}\{y:f(x^{\prime},y)>t\}}\frac{1}{|\nabla_yf|}\phi\left((x_0^{\prime},-\eta_2)\cdot\left(\nabla_1f,\cdots,\nabla_{n-1}f,\nabla_yf\right)\right)d\mathcal{H}^0.\nonumber\\
\end{eqnarray}
Combining (\ref{3o}), (\ref{3p}) and (\ref{3q}) leads to
\begin{eqnarray}\label{3r}
&&\int_{\partial^{\ast}\{y:f^s(x^{\prime},y)>t\}}\frac{1}{|\nabla_yf^s|}\phi\left(x_0^{\prime}\cdot(\nabla_1f^s,\dots,\nabla_{n-1}f^s)+\frac{\eta_1+\eta_2}{2}\cdot\nabla_y f^s\right)d\mathcal{H}^0\nonumber\\
&\leq&\frac{1}{2}\int_{\partial^{\ast}\{y:f(x^{\prime},y)>t\}}\frac{1}{|\nabla_yf|}\phi\left((x_0^{\prime},\eta_1)\cdot\left(\nabla_1f,\cdots,\nabla_{n-1}f,\nabla_yf\right)\right)d\mathcal{H}^0\nonumber\\
&&\;+\frac{1}{2}\int_{\partial^{\ast}\{y:f(x^{\prime},y)>t\}}\frac{1}{|\nabla_yf|}\phi\left((x_0^{\prime},-\eta_2)\cdot\left(\nabla_1f,\cdots,\nabla_{n-1}f,\nabla_yf\right)\right)d\mathcal{H}^0
\end{eqnarray}
for $\mathcal{L}^{n-1}$-a.e. $x^{\prime}\in\pi_{n-1}(\Omega)$ such that $M_f(x^{\prime})>0$ and for $\mathcal{L}^1$-a.e. $t\in (0,M_f(x^{\prime}))$.

Integrating (\ref{3r}) first with respect to $t$ over $(0,M_f(x^{\prime}))$, and then with respect to $x^{\prime}$ over $\pi_{n-1}(\Omega)$, by Fubini's theorem and (\ref{3w}), one gets
\begin{eqnarray}\label{3x}
&&\int_{\pi_{n-1}(\Omega)\times\mathbb{R}_y}\phi\left(\left(x_0^{\prime},\frac{\eta_1+\eta_2}{2}\right)\cdot\nabla f^s(x^{\prime},y)\right)dx^{\prime}dy\nonumber\\
&=&\int_{\pi_{n-1}(\Omega)}dx^{\prime}\int_{\{y:f^s(x^{\prime},y)>0,\;\nabla_yf\neq0\}}\phi\left(\left(x_0^{\prime},\frac{\eta_1+\eta_2}{2}\right)\cdot\nabla f^s(x^{\prime},y)\right)dy\nonumber\\
&=&\int_{\pi_{n-1}(\Omega)}dx^{\prime}\int_{0}^{M_f(x^{\prime})}dt\int_{\partial^{\ast}\{y:f^s(x^{\prime},y)>t\}}\frac{1}{|\nabla_yf^s|}
\phi\left(\left(x_0^{\prime},\frac{\eta_1+\eta_2}{2}\right)\cdot\nabla f^s(x^{\prime},y)\right)d\mathcal{H}^0\nonumber\\
&\leq&\frac{1}{2}\int_{\pi_{n-1}(\Omega)}dx^{\prime}\int_0^{M_f(x^{\prime})}dt\int_{\partial^{\ast}\{y:f(x^{\prime},y)>t\}}
\frac{1}{|\nabla_yf|}\phi\left((x_0^{\prime},\eta_1)\cdot\nabla f(x^{\prime},y)\right)d\mathcal{H}^0\nonumber\\
&&\;+\frac{1}{2}\int_{\pi_{n-1}(\Omega)}dx^{\prime}\int_0^{M_f(x^{\prime})}dt\int_{\partial^{\ast}\{y:f(x^{\prime},y)>t\}}\frac{1}{|\nabla_yf|}\phi
\left((x_0^{\prime},-\eta_2)\cdot\nabla f(x^{\prime},y)\right)d\mathcal{H}^0\nonumber\\
&=&\frac{1}{2}\int_{\pi_{n-1}(\Omega)\times\mathbb{R}_y}\phi\left((x_0^{\prime},\eta_1)\cdot\nabla f(x^{\prime},y)\right)dx^{\prime}dy+\frac{1}{2}\int_{\pi_{n-1}(\Omega)\times\mathbb{R}_y}\phi\left((x_0^{\prime},-\eta_2)\cdot\nabla f(x^{\prime},y)\right)dx^{\prime}dy.\nonumber\\
\end{eqnarray}

Note that here we have made use of (\ref{3w}) and of an analogous equality for $f$.

{\it Step 2:} Let $f$ be any nonnegative function from $W_0^{1,1}(\Omega)$ and let $\omega:=\pi_{n-1}(\Omega)$. Lemma 4.5. in \cite{CF06} ensures that there exists a sequence
$\{f_h\}$ of nonnegative Lipschitz functions, with compact support in $\mathbb{R}^n$, satisfying (\ref{5t})
and converging strongly to $f$ in $W_0^{1,1}(\omega\times\mathbb{R}_y)$. Assume, for a moment, that $\phi$ satisfies
\begin{eqnarray}\label{5u}
0\leq\phi(x)\leq C(1+|x|)\;\;{\rm for}\;x\in\mathbb{R}
\end{eqnarray}
for some positive constant $C$.
 Then for $x\in\mathbb{R}^n$, setting
$$F_1(x):=\phi\left(\left(x_0^{\prime},\frac{\eta_1+\eta_2}{2}\right)\cdot x\right)$$
 $$F_2(x):=\phi\left((x_0^{\prime},\eta_1)\cdot x\right)$$
 and
 $$F_3(x):=\phi\left((x_0^{\prime},-\eta_2)\cdot x\right),$$
 we obtain that $F_1$, $F_2$ and $F_3$ are globally Lipschitz continuous, and hence $F_i(\nabla f_h)$ converges to $F_i(\nabla f)$
in $L^1(\omega\times\mathbb{R}_y)$  for $i=1,2,3$. On the other hand, since Steiner rearrangement is continuous in $W^{1,1}$ (see
e.g. \cite[Theorem 1.]{Bu97}), $f_h^s$ converges to $f^s$ in $W^{1,1}(\omega\times\mathbb{R}_y)$. Thus, by Fatou's lemma and by (\ref{3x}), we get
\begin{eqnarray}\label{5v}
&&\int_{\pi_{n-1}(\Omega)\times\mathbb{R}_y}F_1\left(\nabla f^s(x^{\prime},y)\right)dx^{\prime}dy\nonumber\\
&\leq&\liminf_{h\rightarrow\infty}\int_{\pi_{n-1}(\Omega)\times\mathbb{R}_y}F_1\left(\nabla f_h^s(x^{\prime},y)\right)dx^{\prime}dy\nonumber\\
&\leq&\liminf_{h\rightarrow\infty}\frac{1}{2}\int_{\pi_{n-1}(\Omega)\times\mathbb{R}_y}F_2\left(\nabla f_h(x^{\prime},y)\right)dx^{\prime}dy\nonumber\\
&&+\liminf_{h\rightarrow\infty}\frac{1}{2}\int_{\pi_{n-1}(\Omega)\times\mathbb{R}_y}F_3\left(\nabla f_h(x^{\prime},y)\right)dx^{\prime}dy\nonumber\\
&=&\frac{1}{2}\int_{\pi_{n-1}(\Omega)\times\mathbb{R}_y}F_2\left(\nabla f(x^{\prime},y)\right)dx^{\prime}dy\nonumber\\
&&+\frac{1}{2}\int_{\pi_{n-1}(\Omega)\times\mathbb{R}_y}F_3\left(\nabla f(x^{\prime},y)\right)dx^{\prime}dy.
\end{eqnarray}
 Let us now remove assumption (\ref{5u}). Since $\phi$ is nonnegative
and convex, there exist sequences $\{a_j\}$ of $\mathbb{R}$ and $\{b_j\}$ of $\mathbb{R}$ such that
\begin{eqnarray}
\phi(x)=\sup_{j\in\mathbb{N}}\{a_jx+b_j\}=\sup_{j\in\mathbb{N}}\{(a_jx+b_j)^+\}\;\;{\rm for\;every}\;x\in\mathbb{R}.
\end{eqnarray} Set, for $N\in\mathbb{N}$,
\begin{eqnarray}
\phi_N(x):=\sup_{1\leq j\leq N}\{(a_jx+b_j)^+\}\;\;{\rm for}\;x\in\mathbb{R}.
\end{eqnarray}
Obviously, $\phi_N(x)$ converges monotonically to $\phi(x)$ for every $x\in\mathbb{R}$. Since $\phi_N$ satisfies (\ref{5u}), then (\ref{5v}) holds with $\phi$ replaced by $\phi_N$. Inequality (\ref{5v}) then follows by monotone convergence.

By Step 1 and Step 2, (\ref{3x}) is established for $f\in W_0^{1,1}(\Omega)$. By the definition (\ref{2ff}), (\ref{3x}), (\ref{3t}) and (\ref{3u}), (\ref{3v}) is established.
\end{proof}

\noindent{\bf Proof of Theorem \ref{2.1}}. By Proposition \ref{4f} and (\ref{1a}), Theorem \ref{2.1} is established.

\noindent{\bf Proof of Theorem \ref{2.2}}. By Theorem \ref{2.1}, Remark \ref{19b} and Lemma {\ref{2x}}, Theorem \ref{2.2} is established.\\

Next, we prove Theorem \ref{2.3}.
\begin{lem}\label{L1}
Let $\phi\in\mathcal{N}_s$. Let $\Omega$ be a bounded open subset of $\mathbb{R}^n$ satisfying (\ref{2b}) and let $f$ be a nonnegative function from $W_0^{1,\Phi}(\Omega)$ satisfying (\ref{2i}) and (\ref{2r}).
Then, for $\mathcal{L}^n$-a.e. $(x^{\prime},t)\in \pi_{n-1,n+1}(S_f)^+$, there exist $y_1(x^{\prime},t)$, $y_2(x^{\prime},t)\in\mathbb{R}$ such that $y_1(x^{\prime},t)<y_2(x^{\prime},t)$ and that
\begin{eqnarray}\label{4g}
\{y:\;f(x^{\prime},y)>t\}\;is\;equivalent\;to\;(y_1(x^{\prime},t),y_2(x^{\prime},t)),
\end{eqnarray}
\begin{eqnarray}\label{4i}
\nabla_y f(x^{\prime},y_1(x^{\prime},t))=-\nabla_yf(x^{\prime},y_2(x^{\prime},t))
\end{eqnarray}
and
\begin{eqnarray}\label{4j}
\frac{\nabla_{x^{\prime}} f(x^{\prime},y_1(x^{\prime},t))}{|\nabla_yf(x^{\prime},y_1(x^{\prime},t))|}=
\frac{\nabla_{x^{\prime}}f(x^{\prime},y_2(x^{\prime},t))}{|\nabla_y
f(x^{\prime},y_2(x^{\prime},t))|}+z^{\prime}_0,
\end{eqnarray}
where $z^{\prime}_0\in \mathbb{R}^{n-1}$ is a constant vector.
\end{lem}
\begin{proof}
Assumption (\ref{2r}) ensures that equality necessarily holds in (\ref{3x}).
Thus, equality holds in (\ref{3r}) for $\mathcal{L}^{n-1}$-a.e. $x^{\prime}\in\pi_{n-1}(\Omega)$ and $\mathcal{L}^1$-a.e. $t\in (0,M_f(x^{\prime}))$. Therefore, equality holds both in (\ref{3p}) and in (\ref{3q}) for the same $x^{\prime}$ and $t$.

By Lemma \ref{2q}, $k(x^{\prime},t)=1$ whenever it holds in (\ref{3p}). Thus, by the isoperimetric theorem in $\mathbb{R}$, the set $\{y:f(x^{\prime},y)>t\}$ is equivalent to some interval, say
$(y_1(x^{\prime},t),y_2(x^{\prime},t))$ for $\mathcal{L}^{n-1}$-a.e. $x^{\prime}\in\pi_{n-1}(\Omega)$ and $\mathcal{L}^1$-a.e. $t\in (0,M_f(x^{\prime}))$. Since $f>0$ $\mathcal{L}^n$-a.e. in $\Omega$, then
$\pi_{n-1}(\mathcal{S}_f)^+$ is equivalent to $\pi_{n-1}(\Omega)$. Thus,
\begin{eqnarray}\label{4o}
\pi_{n-1,n+1}(\mathcal{S}_f)^+\;\;{\rm is \;equivalent\;to}\;\;\bigcup_{x^{\prime}\in\pi_{n-1}(\Omega)}\{x^{\prime}\}\times(0,M_f(x^{\prime})).
\end{eqnarray}
Hence (\ref{4g}) follows.

Since $\phi$ is strictly convex, if equality holds in (\ref{3q}) for  $\mathcal{L}^{n-1}$-a.e. $x^{\prime}\in\pi_{n-1}(\Omega)$ and $\mathcal{L}^1$-a.e. $t\in (0,M_f(x^{\prime}))$, then for the same values of $x^{\prime}$ and $t$, we have
\begin{eqnarray}\label{4k}
&&(x_0^{\prime},\eta_1)(\nabla_1f(x^{\prime},y_1(x^{\prime},t)),\cdots,\nabla_{n-1}f(x^{\prime},y_1(x^{\prime},t)),|\nabla_yf(x^{\prime},y_1(x^{\prime},t))|)\nonumber\\
&=&(x_0^{\prime},-\eta_2)(\nabla_1f(x^{\prime},y_2(x^{\prime},t)),\cdots,\nabla_{n-1}f(x^{\prime},y_2(x^{\prime},t)),-|\nabla_yf(x^{\prime},y_2(x^{\prime},t))|)\nonumber\\
\end{eqnarray}
and
\begin{eqnarray}\label{4l}
&&(x_0^{\prime},-\eta_2)(\nabla_1f(x^{\prime},y_1(x^{\prime},t)),\cdots,\nabla_{n-1}f(x^{\prime},y_1(x^{\prime},t)),|\nabla_yf(x^{\prime},y_1(x^{\prime},t))|)\nonumber\\
&=&(x_0^{\prime},\eta_1)(\nabla_1f(x^{\prime},y_2(x^{\prime},t)),\cdots,\nabla_{n-1}f(x^{\prime},y_2(x^{\prime},t)),-|\nabla_yf(x^{\prime},y_2(x^{\prime},t))|).\nonumber\\
\end{eqnarray}
Subtracting (\ref{4l}) from (\ref{4k}), we have
\begin{eqnarray}\label{4m}
(\eta_1+\eta_2)|\nabla_yf(x^{\prime},y_1(x^{\prime},t))|=(\eta_1+\eta_2)|\nabla_yf(x^{\prime},y_2(x^{\prime},t))|.
\end{eqnarray}
Since $\eta_1+\eta_2\neq 0$,
\begin{eqnarray}\label{4n}
|\nabla_yf(x^{\prime},y_1(x^{\prime},t))|=|\nabla_yf(x^{\prime},y_2(x^{\prime},t))|.
\end{eqnarray}

It is easily seen from (\ref{4g}) that for $\mathcal{L}^{n-1}$-a.e. $x^{\prime}\in\pi_{n-1}(\Omega)$ the function $y_1(x^{\prime},\cdot)$ agrees $\mathcal{L}^1$-a.e. in $(0,M_f(x^{\prime}))$ with a nondecreasing function, whereas $y_2(x^{\prime},\cdot)$ agrees $\mathcal{L}^1$-a.e. in $(0,M_f(x^{\prime}))$ with a nonincreasing function in the same interval. Therefore, $f(x^{\prime},\cdot)$ is equivalent to a function (whose level sets are open intervals) which is nondecreasing in $(-\infty,\beta_1(x^{\prime}))$ and nonincreasing in $(\beta_2(x^{\prime}),+\infty)$, where
\begin{eqnarray}
\beta_1(x^{\prime})={\rm ess}\sup\{y_1(x^{\prime},t):\;t<M_f(x^{\prime})\},\;\;\beta_2(x^{\prime})={\rm ess}\inf\{y_2(x^{\prime},t):\;t<M_f(x^{\prime})\}.\nonumber
\end{eqnarray}
Hence, Eq. (\ref{4n}) implies, in fact, (\ref{4i}).

Putting (\ref{4n}) into (\ref{4k}) (or \ref{4l}),  we have
\begin{eqnarray}\label{4b}
(\eta_1-\eta_2)|\nabla_yf(x^{\prime},y_1(x^{\prime},t))|+x_0^{\prime}\cdot\left(\nabla_{x^{\prime}}f(x^{\prime},y_1(x^{\prime},t))-\nabla_{x^{\prime}}f(x^{\prime},y_2(x^{\prime},t))\right)=0
\end{eqnarray}
for $\mathcal{L}^1$-a.e. $t\in (0, M_f(x^{\prime}))$. However (\ref{4b}) says that, for some null set $N^{\prime}\subset (0,M_f(x^{\prime}))$
$$\left\{\frac{\nabla_{x^{\prime}}f(x^{\prime},y_1(x^{\prime},t))-\nabla_{x^{\prime}}f(x^{\prime},y_2(x^{\prime},t))}{|\nabla_{y}f(x^{\prime},y_1(x^{\prime},t))|}:\; t\in(0,M_f(x^{\prime}))\backslash N^{\prime}\right\}$$
is a set of points that must lie in a hyperplane of $\mathbb{R}^{n-1}$ with normal vector $x_0^{\prime}$. But $x_0^{\prime}$ can be chosen in any direction in $\mathbb{R}^{n-1}$, so there exists an $z_0^{\prime}\in\mathbb{R}^{n-1}$ with
$$\left\{\frac{\nabla_{x^{\prime}}f(x^{\prime},y_1(x^{\prime},t))-\nabla_{x^{\prime}}f(x^{\prime},y_2(x^{\prime},t))}{|\nabla_{y}f(x^{\prime},y_1(x^{\prime},t))|}:\; t\in(0,M_f(x^{\prime}))\backslash N^{\prime}\right\}=\{z_0^{\prime}\}.$$
Thus, (\ref{4j}) is established.
\end{proof}

\begin{lem}\cite[Lemma 4.8 and Lemma 4.10]{CF06}\label{L2}
Let $\phi$, $\Omega$ and $f$ be given as in Theorem \ref{2.3}, and
let $y_1(x^{\prime},t)$ and $y_2(x^{\prime},t)$ be defined as in Lemma \ref{L1}. Then there exists a function $b:\pi_{n-1}(\Omega)\rightarrow\mathbb{R}$ such that, for $\mathcal{L}^{n-1}$-a.e. $x^{\prime}\in \pi_{n-1}(\Omega)$,
\begin{eqnarray}\label{4q}
\frac{1}{2}(y_1(x^{\prime},t)+y_2(x^{\prime},t))=b(x^{\prime})\;\; for \;\mathcal{L}^1{\text-}a.e.\;t\in(0,M_f(x^{\prime})).
\end{eqnarray}
Moreover, $b(x^{\prime})\in W^{1,1}_{{\rm loc}}(\pi_{n-1}(\Omega))$.
\end{lem}

\begin{lem}\label{L5}
Let $\phi$, $\Omega$ and $f$ be given as in Theorem \ref{2.3}, and let $b:\pi_{n-1}(\Omega)\rightarrow\mathbb{R}$ be the function defined in Lemma \ref{L2}. If
there exist $z_0^{\prime}\in\mathbb{R}^{n-1}$ and $y_0\in\mathbb{R}$ such that
\begin{eqnarray}
b(x^{\prime})=z_0^{\prime}\cdot x^{\prime}+y_0\;\;for\;\mathcal{L}^{n-1}{\text-}a.e.\;x^{\prime}\in\pi_{n-1}(\Omega),
\end{eqnarray}
then there exist $A\in SL(n)$ and $x_0\in\mathbb{R}^n$ such that
\begin{eqnarray}
f(x)=f^s(Ax+x_0)\;\;for\;\mathcal{L}^{n}{\text-}a.e.\;x^{\prime}\in\Omega.
\end{eqnarray}
\end{lem}
\begin{proof}
Let
\begin{equation}
A=\left(\begin{matrix}
E_{n-1}&0\\
-z_0^{\prime}&1
\end{matrix}\right)\;\;{\rm and}\;\;
x_0=A\left(\begin{matrix}
0\\
-y_0
\end{matrix}\right),
\end{equation}
where $E_{n-1}$ denotes the $(n-1)\times(n-1)$ unit matrix.

Since the level set
$$[f^s(Ax+x_0)]_h=A^{-1}[f^s]_h-A^{-1}x_0$$
and for any $(x^{\prime},y)\in [f^s]_h$,
$$A^{-1}\left(\begin{matrix}
x^{\prime}\\
y
\end{matrix}\right)-A^{-1}x_0=\left(\begin{matrix}
x^{\prime}\\
z_0^{\prime}x^{\prime}+y_0+y
\end{matrix}\right)=\left(\begin{matrix}
x^{\prime}\\
b(x^{\prime})+y
\end{matrix}\right),$$
we have
\begin{eqnarray}\label{2bb}
[f^s(Ax+x_0)]_h=[f]_h\;\;{\rm for\;every}\;h>0.
\end{eqnarray}

By (\ref{2bb}) and the layer cake representation of a non-negative, real-valued measurable function $f$, we get $f(x)=f^s(Ax+x_0)$ for $\mathcal{L}^n$-a.e. $x\in\Omega$.
\end{proof}

{\noindent\bf Proof of Theorem \ref{2.3}.}
By the affine invariance of $\mathcal{E}_{\phi}(f)$, i.e., Lemma \ref{2dd}, the sufficiency is established.  Now we prove the necessity.
Let $y_1(x^{\prime},t)$ and $y_2(x^{\prime},t)$ be defined as in Lemma \ref{L1}, and let $b$ be the function defined as in Lemma \ref{L2}. Let us set
\begin{eqnarray}\label{5l}
z_1(x^{\prime},t)=b(x^{\prime})-\frac{1}{2}\mu_f(x^{\prime},t),\;\;z_2(x^{\prime},t)=b(x^{\prime})+\frac{1}{2}\mu_f(x^{\prime},t)
\end{eqnarray}
for $(x^{\prime},t)\in \pi_{n-1}(\Omega)\times\mathbb{R}_t^{+}$. Then, by Lemma \ref{4e} and Lemma \ref{L2}, $z_i\in BV_{\rm loc}(\pi_{n-1}(\Omega)\times \mathbb{R}^+_t)$, $i=1,2$. By (\ref{4g}) and (\ref{4o}), for $\mathcal{L}^{n-1}$-a.e. $x^{\prime}\in\pi_{n-1}(\Omega)$ and $\mathcal{L}^1$-a.e. $t\in(0,M_f(x^{\prime}))$, we have
\begin{eqnarray}\label{4p}
\mathcal{L}^1(\{y:\;f(x^{\prime},y)>t\})=y_2(x^{\prime},t)-y_1(x^{\prime},t).
\end{eqnarray}
Thus, by (\ref{5l}), (\ref{4q}) and (\ref{4p}), a set $N\subset\pi_{n-1,n+1}(\mathcal{S}_f)^+$ exists such that $$\mathcal{L}^n(\pi_{n-1,n+1}(\mathcal{S}_f)^+\setminus N)=0\;\;{\rm and}\;\;
z_i(x^{\prime},t)=y_i(x^{\prime},t)$$
for $(x^{\prime},t)\in N$. Therefore, thanks to Lemma \ref{L1}, the set $\mathcal{S}_f$ is equivalent to the set $E$ defined by
\begin{eqnarray}
E=\{(x^{\prime},y,t):(x^{\prime},t)\in \pi_{n-1,n+1}(\mathcal{S}_u)^+,\;z_1(x^{\prime},t)<y<z_2(x^{\prime},t)\}.
\end{eqnarray}
Now, define
$$E_1=\{(x^{\prime},y,t):(x^{\prime},t)\in\pi_{n-1}(\Omega)\times \mathbb{R}_t^+,\;y>z_1(x^{\prime},t)\}$$
$$E_2=\{(x^{\prime},y,t):(x^{\prime},t)\in\pi_{n-1}(\Omega)\times \mathbb{R}_t^+,\;y<z_2(x^{\prime},t)\}.$$
Observe that $E$ is equivalent to $E_1\cap E_2$. By Theorem \ref{2s}, the sets $E$, $E_1$ and $E_2$ are of finite perimeter in $U\times\mathbb{R}_y$ for every bounded open set $U\Subset\pi_{n-1}(\Omega)\times\mathbb{R}_t^+$, and hence, by Theorem \ref{2t}, Borel sets $\Omega_E$, $\Omega_{E_1}$ and $\Omega_{E_2}$ exist such that
$$\mathcal{L}^n(\pi_{n-1,n+1}(E)^+\setminus \Omega_E)=0,\;\;\;\mathcal{L}^n((\pi_{n-1}(\Omega)\times\mathbb{R}_t^+)\setminus \Omega_{E_i})=0,\; \;i=1,2,$$
and (\ref{2u})-(\ref{2v}) hold. In particular,
\begin{eqnarray}\label{4r}
(\partial^{\ast} E)_{x^{\prime},t}=\partial^{\ast}(E_{x^{\prime},t})=\{z_1(x^{\prime},t),z_2(x^{\prime},t)\}\;\;{\rm for\;every}\;(x^{\prime},t)\in \Omega_E
\end{eqnarray}
\begin{eqnarray}
(\partial^{\ast} E_i)_{x^{\prime},t}=\partial^{\ast}(E_i)_{x^{\prime},t}=\{z_i(x^{\prime},t)\}\;\;{\rm for\;every}\;(x^{\prime},t)\in \Omega_{E_i},\;i=1,2.
\end{eqnarray}
By Theorem \ref{2w} and by (\ref{2v}) of Theorem \ref{2t}, a Borel set $S$ exists such that $\mathcal{H}^n(S)=0$ and
\begin{eqnarray}
v^E(x^{\prime},y,t)&=&v^{E_i}(x^{\prime},y,t)\nonumber\\
&&{\rm for}\;(x^{\prime},y,t)\in[(\partial^{\ast}E\cap\partial^{\ast} E_i)\setminus S]\cap [(\Omega_E\cap \Omega_{E_i})\times\mathbb{R}_y].
\end{eqnarray}

We next claim that a subset $R$ of $\pi_{n-1,n+1}(E)^+$ exists such that $\mathcal{L}^n(\pi_{n-1,n+1}(E)^+\setminus R)=0$ and
\begin{equation}\label{4s}
\left\{ \begin{aligned}
&\frac{v^E_i(x^{\prime},z_1(x^{\prime},t),t)}{|v^E_y(x^{\prime},z_1(x^{\prime},t),t)|}
=\frac{v_i^E(x^{\prime},z_2(x^{\prime},t),t)}{|v^E_y(x^{\prime},z_2(x^{\prime},t),t)|}
+z^{\prime}_0,\;\;i=1,\dots,n-1,&&\\
&\frac{v^E_y(x^{\prime},z_1(x^{\prime},t),t)}{v^E_t(x^{\prime},z_1(x^{\prime},t),t)}
=\frac{-v_y^E(x^{\prime},z_2(x^{\prime},t),t)}{v_t^E(x^{\prime},z_2(x^{\prime},t),t)},&&
\end{aligned} \right.
\end{equation}
for $(x^{\prime},t)\in R$, where $z^{\prime}_0\in\mathbb{R}^{n-1}$ is a constant vector.
To verify this claim, recall from Theorem \ref{T1} that, since $f\in W_0^{1,1}(\Omega)$, a subset $V$ of $\partial^{\ast}E\cap(\Omega\times\mathbb{R}^+_t)$ exists such that
\begin{eqnarray}\label{4x}
\mathcal{H}^n([\partial^{\ast}E\cap(\Omega\times\mathbb{R}^+_t)]\setminus V)=0
\end{eqnarray}
and
\begin{eqnarray}\label{4u}
v^E(x^{\prime},y,t)=\left(\frac{\nabla_1f(x^{\prime},y)}{\sqrt{1+|\nabla f|^2}},\cdots,\frac{\nabla_{n-1}f(x^{\prime},y)}{\sqrt{1+|\nabla f|^2}},\frac{\nabla_yf(x^{\prime},y)}{\sqrt{1+|\nabla f|^2}},\frac{-1}{\sqrt{1+|\nabla f|^2}}\right)
\end{eqnarray}
for every $(x^{\prime},y,t)\in V$.
Set $Q=\pi_{n-1,n+1}([\partial^{\ast}E\cap(\Omega\times\mathbb{R}_t)]\setminus V)$. Eq. (\ref{4x}) entails that $\mathcal{L}^n(Q)=0$. It is easy to observe that
\begin{eqnarray}\label{4t}
(x^{\prime},z_i(x^{\prime},t),t)\in V\;\;{\rm for}\;\mathcal{L}^n{\text-}a.e.\;(x^{\prime},t)\in\pi_{n-1,n+1}(E)^+\setminus Q.
\end{eqnarray}
Eqs. (\ref{4s}) follow from (\ref{4t}) and (\ref{4u}) and from (\ref{4i})--(\ref{4j}) of Lemma \ref{L1}.

Finally, from Eq. (\ref{2c}) applied to $z_1$ and $z_2$, and from (\ref{4r}) we deduce that a set $T\subset \pi_{n-1}(\Omega)\times \mathbb{R}_t^+$ exists such that $\mathcal{L}^n((\pi_{n-1}(\Omega)\times \mathbb{R}_t^+)\setminus T)=0$ and
\begin{eqnarray}\label{4v}
&&v^{E_i}(x^{\prime},z_i(x^{\prime},t),t)\nonumber\\
&=&(-1)^i\left(\frac{\nabla_1 z_i(x^{\prime},t)}{\sqrt{1+|\nabla z_i|^2}},\cdots,\frac{\nabla_{n-1} z_i(x^{\prime},t)}{\sqrt{1+|\nabla z_i|^2}},\frac{-1}{\sqrt{1+|\nabla z_i|^2}},\frac{\nabla_t z_i(x^{\prime},t)}{\sqrt{1+|\nabla z_i|^2}}\right),\nonumber\\
&&\;i=1,2,
\end{eqnarray}
for $(x^{\prime},t)\in T$. Now, set
$$Z=[\pi_{n-1,n+1}(E)^+\cap N\cap \Omega_E\cap \Omega_{E_1}\cap \Omega_{E_2}\cap R\cap T]\setminus \pi_{n-1,n+1}(S),$$
and note that $\mathcal{L}^n(\pi_{n-1,n+1}(E)^+\setminus Z)=0$. Combining (\ref{4r})--(\ref{4s}) and (\ref{4v}) we infer that
$$\nabla_{x^{\prime}}z_1(x^{\prime},t)+\nabla_{x^{\prime}}z_2(x^{\prime},t)=-z^{\prime}_0$$
and
$$\nabla_{t}z_1(x^{\prime},t)+\nabla_{t}z_2(x^{\prime},t)=0$$
for $(x^{\prime},t)\in Z$, and hence for $\mathcal{L}^n$-a.e. $(x^{\prime},t)\in\pi_{n-1,n+1}(\mathcal{S}_u)^+$. Consequently,
\begin{eqnarray}\label{4w}
\nabla_{x^{\prime}}b(x^{\prime})=-\frac{1}{2}z^{\prime}_0\;{\rm for}\;\mathcal{L}^{n-1}{\text-}{\rm a.e.}\;\;x^{\prime}\in\pi_{n-1}(\Omega).
\end{eqnarray}
 Thus, since $b\in W_{\rm loc}^{1,1}(\pi_{n-1}(\Omega))$ and satisfies (\ref{4w}), and $\pi_{n-1}(\Omega)$ is assumed to be connected, then a constant $y_0\in\mathbb{R}$ exists such that (see e.g. \cite[Corollary 2.1.9]{Ziemer89})
 \begin{eqnarray}\label{3a}
 b(x^{\prime})=-\frac{1}{2}z^{\prime}_0\cdot x^{\prime}+y_0\;{\rm for }\; \mathcal{L}^{n-1}{\text-}a.e.\; x^{\prime}\in\pi_{n-1}(\Omega).
 \end{eqnarray}
  By Lemma \ref{L5}, the necessity is established.
 \qed

\section{Proof of Theorem \ref{2.4}}\label{s5}

For $u\in S^{n-1}$, let $u^{\perp}$ denote the $n$-dimensional linear subspace orthogonal to $u$ in $\mathbb{R}^n$.  For a Lebesgue measurable set $E\subset\mathbb{R}^n$ and $x^{\prime}\in u^{\perp}$, let
\begin{eqnarray}\label{5.1}
E_{x^{\prime},u}:=\{x^{\prime}+su:\;s\in\mathbb{R}\}\cap E.
\end{eqnarray}
Let $F\subset \mathbb{R}^{n+1}$ denote a bounded Lebesgue measurable set. Let $\pi_u(F)$ denote the orthogonal projection of $F$ onto $u^{\perp}$ and let $\pi_{u,t}(F)$ denote the orthogonal projection of $F$ onto $u^{\perp}\times\mathbb{R}_t$.
Similar to \eqref{5.1}, for $u\in S^{n-1}$ and $(x^{\prime},t)\in u^{\perp}\times\mathbb{R}_t$, let
\begin{eqnarray}
F_{(x^{\prime},t),u}:=\{(x^{\prime},t)+su:\;s\in\mathbb{R}\}\cap F.
\end{eqnarray}

 For $u\in S^{n-1}$ and $K\subset\mathbb{R}^n$, let
$$\pi_u(K)\times \mathbb{R}_u:=\{x^{\prime}+su:x^{\prime}\in\pi_u(K),\;s\in\mathbb{R}\}.$$
For $u\in S^{n-1}$ and $f\in W_0^{1,\Phi}(\Omega)$, let $f_u^s$ denote the Steiner symmetrization of $f$ with respect to $u$.
For fixed $x^{\prime}\in\pi_u(\Omega)^+$, let
$$M_u(x^{\prime}):={\rm ess}\sup\{f(x^{\prime}+su):\;x^{\prime}+su\in\Omega\}$$
and
\begin{eqnarray}\label{5b}
D_{f,u}:&=&\{x^{\prime}+su\in\Omega:\;x^{\prime}\in\pi_u(\Omega),\;\; \nabla_u f(x^{\prime}+su)=0\}\nonumber\\
&&\cap\{x^{\prime}+su\in\Omega:\;x^{\prime}\in\pi_u(\Omega),\;\;M_u(x^{\prime})=0\;{\rm or}\;f(x^{\prime}+su)<M_f(x^{\prime})\}.\nonumber\\
\end{eqnarray}

\begin{lem}\label{5q}
For a bounded Lebesgue measurable set $E\subset\mathbb{R}^n$ let
\begin{eqnarray}\label{5p}
E_1:=\left\{x\in \mathbb{R}^n:\;\lim_{\varepsilon\rightarrow0^+}\frac{\mathcal{L}^n(E\cap C(x,\varepsilon))}{\mathcal{L}^n(C(x,\varepsilon))}=1\right\},
\end{eqnarray}
where $C(x,\varepsilon)$ is a cube centered at $x$ and whose side length is $2\varepsilon$.
 Then
\begin{eqnarray}
\mathcal{L}^n(E_1\triangle E)=0.
\end{eqnarray}
\end{lem}
\begin{proof}
We will use Lebesgue's density theorem:  If $A$ is a Lebesgue measurable subset of $\mathbb{R}^n$ and
\begin{eqnarray}\label{6h}
\bar{A}:=\left\{x\in A:\;\lim_{\varepsilon\rightarrow0^+}\frac{\mathcal{L}^n(A\cap
C(x,\varepsilon))}{\mathcal{L}^n(C(x,\varepsilon))}=1\right\},
\end{eqnarray}
then $\mathcal{L}^n(A\setminus \bar{A})=0$.

Let $E^c$ denote the complement of $E$. Let $\overline{E^c}$ and $\bar{E}$ be the sets defined as in (\ref{6h}). On the one hand, if $x\in E_1\setminus E$, then $x\in E^c\setminus \overline{E^c}$. Since $\mathcal{L}^n(E^c\setminus \overline{E^c})=0$, $\mathcal{L}^n(E_1\setminus E)=0$. On the other hand, if $x\in E\setminus E_1$, then $x\in E\setminus \bar{E}$. Since $\mathcal{L}^n(E\setminus \bar{E})=0$, $\mathcal{L}^n(E\setminus E_1)=0$. In summary, $\mathcal{L}^n(E_1\triangle E)=0$.
\end{proof}

The following lemma was proved in Lemma 2.2 of the paper \cite{CFNT17}, here we give a different proof.
\begin{lem}\label{5c}
Let $E\subset\mathbb{R}^n$ be a bounded measurable set. If there is a dense set $T$ of directions in $S^{n-1}$ such that for every $u\in T$, for $\mathcal{L}^{n-1}$-a.e. $x^{\prime}\in \pi_u(E)$, $E_{x^{\prime},u}$ is equivalent to a closed line segment, then $E$ is a convex body up to an $\mathcal{L}^n$-negligible set.
\end{lem}
\begin{proof}
Let $E_1$ be defined as in (\ref{5p}). By Lemma \ref{5q}, we only need to prove that $E_1$ is a convex set. Suppose that $E_1$ is not  convex, then there exist $x_1,x_2\in E_1$ such that there exists a point $z\in (x_1,x_2)$ but $z\notin E_1$.

 By (\ref{5p}), there exist $0<\varepsilon_0<1$ and a sequence of $\varepsilon_i>0$, $i=1,2,\dots$, such that
\begin{eqnarray}\label{3c}
\lim_{i\rightarrow\infty}\varepsilon_i=0
\end{eqnarray}
and for any $i$,
\begin{eqnarray}\label{3g}
\frac{\mathcal{L}^n(E\cap C(z,\varepsilon_i))}{\mathcal{L}^n(C(z,\varepsilon_i))}\leq 1-\varepsilon_0
\end{eqnarray}
and there exits $i_0$ such that $i\geq i_0$
\begin{eqnarray}\label{3l}
\frac{\mathcal{L}^n(E\cap C(x_k,\varepsilon_{i}))}{\mathcal{L}^n(C(x_k,\varepsilon_{i}))}\geq1-\frac{\varepsilon_0}{12},\;\;k=1,2,
\end{eqnarray}
where $\bar{w}:=\frac{x_2-x_1}{|x_2-x_1|}$ is a common normal vector of one of the $(n-1)$-dimensional faces of  $C(z,\varepsilon_i)$, $C(x_1,\varepsilon_i)$ and $C(x_2,\varepsilon_i),$ and the orthogonal projections of $C(z,\varepsilon_i)$, $C(x_1,\varepsilon_i)$ and $C(x_2,\varepsilon_i)$ onto $\bar{w}^{\perp}$ are same.
\begin{figure}[htb]
\centering
  \includegraphics[height=7cm]{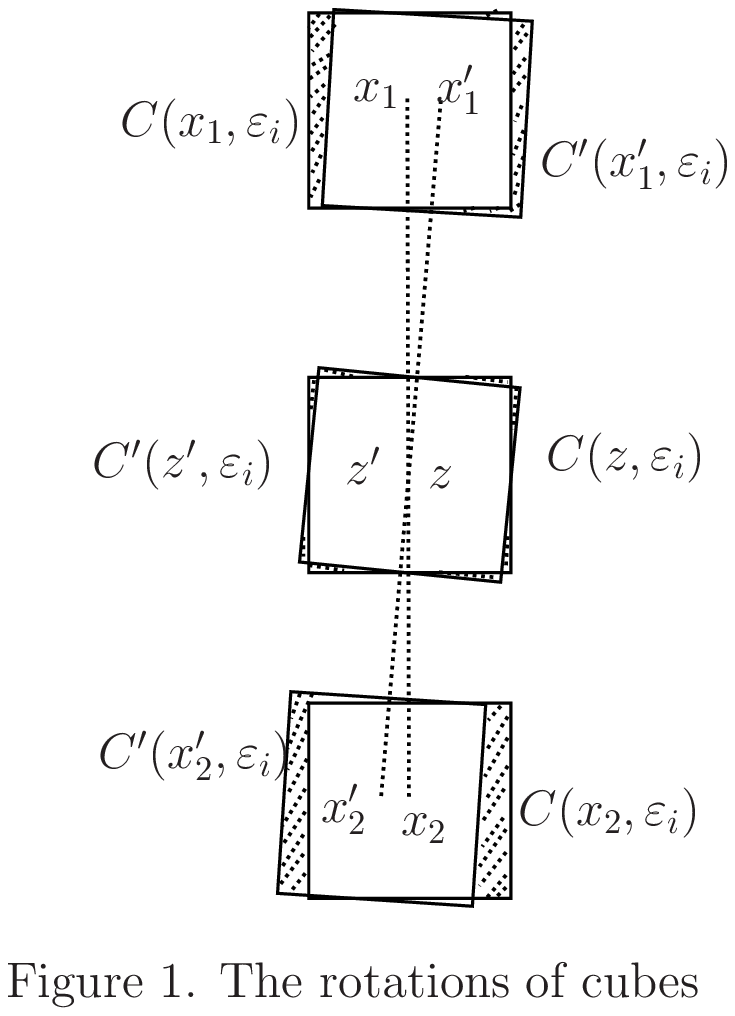}
\end{figure}

For $i\geq i_0$, we aim to rotate   $C(z,\varepsilon_i)$, $C(x_1,\varepsilon_i)$ and $C(x_2,\varepsilon_i)$ around $z,$ such that: $z^{\prime}$, $x_1^{\prime}$ and $x_2^{\prime}$ are collinear, $z^{\prime}=z$ and $w^{\prime}:=\frac{x^{\prime}_2-x^{\prime}_1}{|x^{\prime}_2-x^{\prime}_1|}$ is the common normal vector of the corresponding $(n-1)$-dimensional faces of  $C^{\prime}(z^{\prime},\varepsilon_i)$, $C^{\prime}(x_1^{\prime},\varepsilon_i)$ and $C^{\prime}(x_2^{\prime},\varepsilon_i)$ and the projections of $C^{\prime}(z^{\prime},\varepsilon_i)$, $C^{\prime}(x_1^{\prime},\varepsilon_i)$ and $C^{\prime}(x_2^{\prime},\varepsilon_i)$ onto $(w^{\prime})^{\perp}$ are same (see Figure 1, where dashed areas denote symmetric differences).

Actually, by the  denseness of $T$ and the continuity of  the Lebesgue measure, there exists a rotation $\Psi \in SO(n),$ such that the points
$$x_1' = z+ \Psi(x_1-z),\quad \quad x_2' = z+ \Psi(x_2-z), \quad \quad w^{\prime}:=\frac{x^{\prime}_2-x^{\prime}_1}{|x^{\prime}_2-x^{\prime}_1|}$$
satisfy $\omega'\in T$ and
\begin{eqnarray}\label{5w}
\mathcal{L}^n(C(z,\varepsilon_i)\triangle C^{\prime}(z^{\prime},\varepsilon_i))\leq \frac{\varepsilon_0}{2}\mathcal{L}^n(C^{\prime}(z^{\prime},\varepsilon_i)),
\end{eqnarray}
\begin{eqnarray}\label{5x}
\mathcal{L}^n(C(x_k,\varepsilon_i)\triangle C^{\prime}(x_k^{\prime},\varepsilon_i))\leq \frac{\varepsilon_0}{12}\mathcal{L}^n(C^{\prime}(x^{\prime}_k,\varepsilon_i)),\;\;k=1,2.
\end{eqnarray}
By (\ref{3g}) and (\ref{5w}), we have
\begin{eqnarray}\label{5y}
\mathcal{L}^n(C^{\prime}(z^{\prime},\varepsilon_i)\cap E)&=&\mathcal{L}^n(C^{\prime}(z^{\prime},\varepsilon_i)\cap(C(z,\varepsilon_i)\cup C(z,\varepsilon_i)^c)\cap E)\nonumber\\
&=&\mathcal{L}^n(\left[\left(C^{\prime}(z^{\prime},\varepsilon_i)\cap C(z,\varepsilon_i)\right)\cup\left( C^{\prime}(z^{\prime},\varepsilon_i)\cap C(z,\varepsilon_i)^c\right)\right]\cap E)\nonumber\\
&\leq&\mathcal{L}^n\left(C(z,\varepsilon_i)\cap E\right)+\mathcal{L}^n\left(C^{\prime}(z^{\prime},\varepsilon_i)\cap C(z,\varepsilon_i)^c\right)\nonumber\\
&\leq&\left(1-\frac{\varepsilon_0}{2}\right)\mathcal{L}^n(C^{\prime}(z^{\prime},\varepsilon_i)).
\end{eqnarray}
Thus,
\begin{eqnarray}\label{6a}
\mathcal{L}^n(C^{\prime}(z^{\prime},\varepsilon_i)\setminus E)\geq \frac{\varepsilon_0}{2}\mathcal{L}^n(C^{\prime}(z^{\prime},\varepsilon_i)).
\end{eqnarray}
Similarly, by (\ref{3l}) and (\ref{5x}), for $k=1,2$, we have
\begin{eqnarray}\label{5z}
\mathcal{L}^n(C^{\prime}(x_k^{\prime},\varepsilon_i)\cap E)&=&\mathcal{L}^n(C^{\prime}(x_k^{\prime},\varepsilon_i)\cap(C(x_k,\varepsilon_i)\cup C(x_k,\varepsilon_i)^c)\cap E)\nonumber\\
&=&\mathcal{L}^n(\left[\left(C^{\prime}(x_k^{\prime},\varepsilon_i)\cap C(x_k,\varepsilon_i)\right)\cup\left( C^{\prime}(x_k^{\prime},\varepsilon_i)\cap C(x_k,\varepsilon_i)^c\right)\right]\cap E)\nonumber\\
&=&\mathcal{L}^n\left(C^{\prime}(x_k^{\prime},\varepsilon_i)\cap C(x_k,\varepsilon_i)\cap E\right)+\mathcal{L}^n\left(C^{\prime}(x_k^{\prime},\varepsilon_i)\cap C(x_k,\varepsilon_i)^c\cap E\right)\nonumber\\
&=&\mathcal{L}^n\left(C(x_k,\varepsilon_i)\cap E\right)-\mathcal{L}^n\left(C^{\prime}(x_k^{\prime},\varepsilon_i)^c\cap C(x_k,\varepsilon_i)\cap E\right)\nonumber\\
&&+\mathcal{L}^n\left(C^{\prime}(x_k^{\prime},\varepsilon_i)\cap C(x_k,\varepsilon_i)^c\cap E\right)\nonumber\\
&\geq&\mathcal{L}^n\left(C(x_k,\varepsilon_i)\cap E\right)-\mathcal{L}^n\left((C^{\prime}(x_k^{\prime},\varepsilon_i)\triangle C(x_k,\varepsilon_i))\cap E\right)\nonumber\\
&\geq&\left(1-\frac{\varepsilon_0}{6}\right)\mathcal{L}^n(C^{\prime}(x^{\prime}_k,\varepsilon_i)).
\end{eqnarray}

Let $w^{\prime}\mathbb{R}:=\{rw^{\prime}:r\in\mathbb{R}\}$ and
\begin{eqnarray} 
D_0:=\{x^{\prime}\in\pi_{w^{\prime}}(C^{\prime}(z^{\prime},\varepsilon_i)):\;\mathcal{L}^1\left((x^{\prime}+w^{\prime}\mathbb{R})\cap(C^{\prime}(z^{\prime},\varepsilon_i)\setminus E)\right)>0\}.
\end{eqnarray}
Then $\mathcal{L}^{n-1}(D_0)\geq\frac{\varepsilon_0}{2}\mathcal{L}^{n-1}(\pi_{w^{\prime}}(C^{\prime}(z^{\prime},\varepsilon_i)))$. Otherwise, $\mathcal{L}^n(C^{\prime}(z^{\prime},\varepsilon_i)\setminus E)<\frac{\varepsilon_0}{2}\mathcal{L}^n(C^{\prime}(z^{\prime},\varepsilon_i))$, a contradiction.
Thus, there exists $D_1\subset D_0$ such that $\mathcal{L}^{n-1}(D_1)>0$ and for any $x^{\prime}\in D_1$,
\begin{eqnarray}\label{6c}
\mathcal{L}^1((x^{\prime}+w^{\prime}\mathbb{R})\cap C^{\prime}(x_1^{\prime},\varepsilon_i)\cap E)>0\;\;{\rm and}\;\;\mathcal{L}^1((x^{\prime}+w^{\prime}\mathbb{R})\cap C^{\prime}(x_2^{\prime},\varepsilon_i)\cap E)>0.
 \end{eqnarray}
 Otherwise, if for $\mathcal{L}^{n-1}$-a.e. $x^{\prime}\in D_0$, either $\mathcal{L}^1((x^{\prime}+w^{\prime}\mathbb{R})\cap C^{\prime}(x_1^{\prime},\varepsilon_i)\cap E)=0$ or $\mathcal{L}^1((x^{\prime}+w^{\prime}\mathbb{R})\cap C^{\prime}(x_2^{\prime},\varepsilon_i)\cap E)=0$, then
\begin{eqnarray}
\mathcal{L}^n((C^{\prime}(x_1^{\prime},\varepsilon_1)\cup C^{\prime}(x_2^{\prime},\varepsilon_1))\setminus E)\geq \frac{\varepsilon_0}{2}\mathcal{L}^n(C^{\prime}(x_1^{\prime},\varepsilon_i)).
\end{eqnarray}
Therefore, we have
\begin{eqnarray}
&&\mathcal{L}^n(C^{\prime}(x_1^{\prime},\varepsilon_i)\cap E)+ \mathcal{L}^n(C^{\prime}(x_2^{\prime},\varepsilon_i)\cap E)\nonumber\\
&=&\mathcal{L}^n((C^{\prime}(x_1^{\prime},\varepsilon_i)\cap E)\cup (C^{\prime}(x_2^{\prime},\varepsilon_i)\cap E))\nonumber\\
&=&\mathcal{L}^n((C^{\prime}(x_1^{\prime},\varepsilon_i)\cup C^{\prime}(x_2^{\prime},\varepsilon_i))\cap E)\nonumber\\
&\leq& (2-\frac{\varepsilon_0}{2})\mathcal{L}^n(C^{\prime}(x_1^{\prime},\varepsilon_i)),
\end{eqnarray}
which contradicts
\begin{eqnarray}
\mathcal{L}^n(C^{\prime}(x_1^{\prime},\varepsilon_i)\cap E)+ \mathcal{L}^n(C^{\prime}(x_2^{\prime},\varepsilon_i)\cap E)\geq (2-\frac{\varepsilon_0}{3})\mathcal{L}^n(C^{\prime}(x_1^{\prime},\varepsilon_i)).
\end{eqnarray}

Since (\ref{6a}), (\ref{6c}) and $\mathcal{L}^{n-1}(D_1)>0$ contradict to the assumptions, $E_1$ must be   convex.
\end{proof}

\begin{lem}\label{5e} Let $K\in\mathcal{K}^n$ be a convex body. If there is a dense set $T$ of directions in $S^{n-1}$ such that for each $u\in T$, the midpoints of chords of $K$ parallel to $u$ lie in an affine subspace of $\mathbb{R}^n$, then for any $u\in S^{n-1}$, the midpoints of chords of $K$ parallel to $u$ lie in an affine subspace of $\mathbb{R}^n$.
\end{lem}
\begin{proof}
For fixed $x\in {\rm int}K$, let
$$\rho(x,u):=\max\{r>0:\;x+ru\in K\}$$
 denote the radial function of $K$ with respect to $x$. It is clear that $\rho(x,u)$ is continuous with respect to $u\in S^{n-1}$. Therefore, the midpoint of $(x+u\mathbb{R})\cap K$, denoted by $m(x,u)$, satisfies that
$$m(x,u)=x+\frac{\rho(x,u)-\rho(x,-u)}{2}u$$
and $m(x,u)$ is continuous with respect to $u\in S^{n-1}$. By the denseness of $T$, there exists a sequence of vectors $\{u_i\}_{i=1}^{\infty}\subset T$ such that $\lim_{i\rightarrow \infty}u_i=u_0$. By the assumptions, there exists a sequence of vectors $\{v_i\}_{i=1}^{\infty}\subset T$ and $\alpha_i\in\mathbb{R}$ such that for any $x\in {\rm int}K$,
\begin{eqnarray}\label{5i}
(m(x,u_i),-1)\cdot (v_i,\alpha_i)=0.
\end{eqnarray}
Since $v_i\in S^{n-1}$ and $|\alpha_i|<\sup\{h_K(u):\;u\in S^{n-1}\}$, there exist convergent subsequence $\{v_{i_j}\}$ of $\{v_i\}$ and $\{\alpha_{i_j}\}$ of $\alpha_i$ such that
\begin{eqnarray}\label{5j}
\lim_{j\rightarrow\infty}v_{i_j}=v_0\;\;{\rm and}\;\;\lim_{j\rightarrow\infty}\alpha_{i_j}=\alpha_0.
\end{eqnarray}
By (\ref{5i}) and (\ref{5j}), we have for any $x\in{\rm int}K$,
\begin{eqnarray}
(m(x,u_0),-1)\cdot (v_0,\alpha_0)=0.
\end{eqnarray}
Since $x\in {\rm int}K$ is arbitrary, the midpoints of chords of $K$ parallel to $u_0$ lie in an affine subspace of $\mathbb{R}^n$.
\end{proof}

By a classical characterization of ellipsoids (see \cite[Theorem 10.2.1]{Schneider13}) and Lemma \ref{5e}, we obtain the following theorem of characterization of ellipsoids.
\begin{thm}\label{5f}(Characterization of ellipsoids.)  A convex body $K\in\mathcal{K}^n$ is an ellipsoid if and only if there exists a dense set $T$ of directions in $S^{n-1}$ such that for each $u\in T$, the midpoints of chords of $K$ parallel to $u$ lie in an affine subspace of $\mathbb{R}^n$.
\end{thm}
\begin{lem}\label{5g}
Let $E_1,E_2\subset\mathbb{R}^n$ be origin-centered ellipsoids. If there exists a dense set $T$ of directions in $S^{n-1}$ such that for all $u\in T$, the midpoints of chords of $E_1$ and $E_2$ parallel to $u$ lie in the same hyperplane, then there exists $r>0$ such that
$$E_1=rE_2.$$
\end{lem}

\begin{proof}
First, we prove that for any $u_0\in S^{n-1}$, the midpoints of chords of $E_1$ and $E_2$ parallel to $u_0$ lie in a hyperplane. Since $T$ is dense in $S^{n-1}$, there exists a sequence of vectors $\{u_i\}_{i=1}^{\infty}\subset T$ such that $\lim_{i\rightarrow \infty}u_i=u_0$. By assumptions and the proof of Lemma \ref{5e}, there exist $v_0$ and $\alpha_0$ such that for any
$$x\in\{z\in\mathbb{R}^n:\;z\;{\rm is\;the\;midpoint\;of\;chord\;of}\;E_i\;{\rm parallel\;to}\;u_0,\;i=1,2\},$$
\begin{eqnarray}
(x,-1)\cdot(v_0,\alpha_0)=0.
\end{eqnarray}
Thus, the midpoints of chords of $E_1$ and $E_2$ parallel to $u_0$ lie in a hyperplane.
By the arbitrariness of $u_0\in S^{n-1}$ and \cite[Lemma 5.3]{LYJ17},  there exists $r>0$ such that $E_1=rE_2$.
\end{proof}

\begin{lem}\label{5h}\label{16b}
Let $\Omega$ be a bounded connected open subset of $\mathbb{R}^n$. Let $f\in W_0^{1,\Phi}(\Omega)$ be a nonnegative function fulfilling (\ref{2z}). Then there exist $A\in SL(n)$ and $x_0\in\mathbb{R}^n$ such that
\begin{eqnarray}
f(x)=f^{\star}(Ax+x_0)\;\;for\;\mathcal{L}^{n}{\text-}a.e.\;x\in\Omega
\end{eqnarray}
if and only if there exists a dense set $T$ of directions in $S^{n-1}$ and for any $u\in T$, the following statements hold:

(i). for $\mathcal{L}^n$-a.e. $(x^{\prime},t)\in\pi_{u,t}(\mathcal{S}_f)^+$, $(\mathcal{S}_f)_{(x^{\prime},t),u}$ is equivalent to a closed line segment;

(ii). the midpoints of all closed line segments obtained in (i) lie in an affine subspace of $\mathbb{R}^{n+1}$ parallel to $e_{n+1}$.
\end{lem}
\begin{proof}
The necessity of the conditions is clear. Now we prove the sufficiency. On the one hand, by (i) and  Lemma \ref{5c}, the level set $[f]_h$ is a convex body up to an $\mathcal{L}^n$-negligible set for $\mathcal{L}^1$-a.e. $h\in (0,{\rm ess}\sup f)$.  Therefore, by (ii), the arbitrariness of $u\in T$ and  Theorem \ref{5f},  the level set $[f]_h$ is an ellipsoid up to an $\mathcal{L}^n$-negligible set for $\mathcal{L}^1$-a.e. $h\in (0,{\rm ess}\sup f)$. By (ii), the arbitrariness of $u\in T$ and  Lemma \ref{5g}, the level sets $[f]_{h_1}$ and $[f]_{h_2}$ are homothetic ellipsoids with a common center up to an $\mathcal{L}^n$-negligible set for $\mathcal{L}^1$-a.e. $h_1,h_2\in (0,{\rm ess}\sup f)$.

Therefore, there exist $A\in SL(n)$ and $x_0\in\mathbb{R}^n$ such that for $\mathcal{L}^1$-a.e. $h\in (0,{\rm ess}\sup f)$,
\begin{eqnarray}\label{8d}
[f]_h=\left(\frac{|[f]_h|}{\omega_n}\right)^{\frac{1}{n}}A^{-1}B_n-A^{-1}x_0,
\end{eqnarray}
up to an $\mathcal{L}^n$-negligible set.

On the other hand, by the definition of  Schwarz spherical symmetrization, for $h\in(0,{\rm ess}\sup f)$, $[f^{\ast}(Ax+x_0)]_h$ is also an ellipsoid and
\begin{eqnarray}\label{8e}
[f^{\ast}(Ax+x_0)]_h=\left(\frac{|[f]_h|}{\omega_n}\right)^{\frac{1}{n}}A^{-1}B_n-A^{-1}x_0.
\end{eqnarray}
By (\ref{8d}), (\ref{8e}) and the layer cake representation of a non-negative, real-valued measurable function $f$, we get $f(x)=f^{\star}(Ax+x_0)$ for $\mathcal{L}^n$-a.e. $x\in\Omega$.
\end{proof}

\begin{lem}\label{7d}
Let $D\subset \mathbb{R}^n$ be a bounded measurable set satisfying $\mathcal{L}^n(D)>0$. If there exists  an uncountable subset $I$ of $\mathbb{R}$ such that $A_i\subset D$ and $\mathcal{L}^n(A_i)>0$ for every $i\in I$,  then there exist $i,j\in I$ and $i\neq j$ such that $\mathcal{L}^n(A_{i}\cap A_j)>0$.
\end{lem}
\begin{proof} Suppose $\mathcal{L}^n(A_{i}\cap A_j)=0$ for any $i,j\in I$ and $i\neq j$. On the one hand, since $\mathcal{L}^n(D)$ is finite, for any positive integer $k$, $\{i\in I:\;\mathcal{L}^n(A_{i})>\frac{1}{2^k}\}$ is finite. On the other hand, for any  $i\in I$, there exists a positive integer $k_0$ such that $\mathcal{L}^n(A_i)\geq \frac{1}{2^{k_0}}$. Therefore, $I$ is countable,  a contradiction.
\end{proof}

\begin{lem}\label{7g}
Let $(X,\Sigma,\mu)$ be a  measure space with $\mu(X)<\infty$. Let $D_u\in \Sigma$ be  a  measurable subset of $X$ for every $u\in S^{n-1}$. If there exists a Borel set $S\subset S^{n-1}$ such that $\mathcal{H}^{n-1}(S)>0$ and  $\mu(D_u)>0$ for every $u\in S$, then there exist $n$ linearly independent vectors $u_1,u_2,\dots,u_n\in S$ such that
\begin{eqnarray}\label{6d}
\mu\left(\bigcap_{i=1}^{n}D_{u_i}\right)>0.
\end{eqnarray}
\end{lem}
\begin{proof}
Step 1: Since $\mu$ is a finite measure, we observe the basic property in measure theory that  any index set $I$ satisfying $\mu(D_i)>0$ and
\[\mu(D_i\cap D_j)=0, ~ \forall i,j\in I, i\neq j,\]
is countable.

Step 2: We will prove a more general result by induction: for $k=1,\ldots,n$ and $H_k$ is a $k$-dimensional linear subspace of $\mathbb{R}^n$, if
\begin{eqnarray}
S_k\subset S^{n-1}\cap H_k \;\;{\rm and}\;\;\mathcal{H}^{k-1}(S_k)>0,
 \end{eqnarray}
 and
 \begin{eqnarray}
 \mu(D_{u})>0\;\;{\rm for\;every}\;u\in S_k,
 \end{eqnarray}
 then there are $k$ linearly independent vector $u_1,u_2,\ldots,u_k\in S_k$ such that
 \begin{eqnarray}\label{6i}
\mu\left(\bigcap_{i=1}^{k}D_{u_i}\right)>0.
\end{eqnarray}
If $k=1$, this conclusion is clear. Suppose $k=2$. $\mathcal{H}^1(S_2)>0$ implies that $S_2$ is uncountable. By Step 1, we can find two linearly independent vectors $u^2_1,u^2_2\in S_2$ such that
\begin{eqnarray}
\mu\left(\bigcap_{i=1}^{2}D_{u^2_i}\right)>0.
\end{eqnarray}
Assume that (\ref{6i}) is established for $k=m$, and $m\in \{2,\ldots,n-1\}$, i.e. there exist $m$ linearly independent vectors $u^m_1,\ldots,u^m_{m}\in S_m$ such that
\begin{eqnarray}\label{6j}
\mu\left(\bigcap_{i=1}^{m}D_{u^m_i}\right)>0.
\end{eqnarray}

Next, we consider the case $k=m+1$. Since $S_{m+1}\subset S^{n-1}\cap H_{m+1}$ and $\mathcal{H}^{m}(S_{m+1})>0$, for any $2$-dimensional linear  subspace $\bar{H}_2$ of $H_{m+1}$, there exists a set $\delta\subset S^{n-1}\cap \bar{H}_2$ such that $\mathcal{H}^1(\delta)>0$ and
\begin{eqnarray}
\mathcal{H}^{m-1}(S_{m+1}\cap v^{\perp})>0\;\;{\rm for\;every}\;v\in\delta.
\end{eqnarray}
The subset $\delta$ can be obtained by using the generalized spherical coordinates to compute $\mathcal{H}^{m}(S_{m+1})$.

By the assumption (\ref{6j}), for any $v\in\delta$, there exist $m$ linearly independent vectors $u^m_{v,1},\ldots,u^m_{v,m}\in S_{m+1}\cap v^{\perp}$ such that
\begin{eqnarray}\label{6k}
\mu\left(\bigcap_{i=1}^{m}D_{u^m_{v,i}}\right)>0.
\end{eqnarray}
By $\mathcal{H}^1(\delta)>0$ and Step 1, there are two different $v_1,v_2\in\delta$ satisfying (\ref{6k}) and
\begin{eqnarray}\label{6l}
\mu\left(\left(\bigcap_{i=1}^{m}D_{u^m_{v_1,i}}\right)\cap \left(\bigcap_{i=1}^{m}D_{u^m_{v_2,i}}\right)\right)>0.
\end{eqnarray}
Since $u^m_{v_1,i}$ and $u^m_{v_2,i}$, $i=1,\ldots,m$, lie in two different subspaces $v_1^{\perp}$ and $v_2^{\perp}$, there must be a $u^{m}_{v_2,i_0}$ such that $u^m_{v_1,1},\ldots,u^m_{v_1,m}$ and $u^m_{v_2,i_0}$ are $m+1$ linearly independent vectors. By (\ref{6l})
 \begin{eqnarray}
\mu\left(\bigcap_{i=1}^{m}D_{u^m_{v_1,i}}\cap D_{u^{m}_{v_2,i_0}}\right)>0.
\end{eqnarray}

By induction with respect to the dimension, (\ref{6i}) is established for $k=1,\ldots,n$. Let $S=S_n$, we get the desired result.
\end{proof}

\begin{lem}\label{L4}
Let $\Omega$ be a  set of finite perimeter in $\mathbb{R}^n$ and let
\begin{eqnarray}
D_u =\{x\in\partial^{\ast}\Omega:u\cdot v^{\Omega}(x)=0\}.
\end{eqnarray}
Then there exists a set $T_1\subset S^{n-1}$ such that $\mathcal{H}^{n-1}(S^{n-1}\setminus T_1)=0$ and $\mathcal{H}^{n-1}(D_u)=0$ for any $u\in T_1$.
\end{lem}

\begin{proof} Let
$$T_2:=\{u\in  S^{n-1}:\;\mathcal{H}^{n-1}(D_u)>0\}.$$
We would like to apply Lemma \ref{7g}, and take $\mu$ to be the restriction of $\mathcal{H}^{n-1}$ to $\partial^{\ast}\Omega$.  If $\mathcal{H}^{n-1}(T_2)>0$, then  there exist $n$ linearly independent vectors $u_1,u_2,\dots,u_n\in T_2$ such that $\mathcal{H}^{n-1}(\bigcap_{i=1}^{n}D_{u_i})>0$. Let $D=\bigcap_{i=1}^{n}D_{u_i}$. Then for any $x\in D$, we have
$$u_i\cdot v^{\Omega}(x)=0,\;{\rm for\;}i=1,\dots,n.$$
Thus $v^{\Omega}(x)=0$. This is a contradiction, since $|v^{\Omega}(x)|=1$. Let $T_1:=S^{n-1}\setminus T_2$. Then $T_1$ satisfies the conclusion.
\end{proof}

{\bf \noindent Proof of Theorem \ref{2.4}.} On the one hand, if $f(x)=f^{\star}(Ax+x_0)$ for $\mathcal{L}^n$-a.e. $x\in\Omega$ with $A\in SL(n)$ and $x_0\in\mathbb{R}^n$, by the the affine invariance of $\mathcal{E}_{\phi}(f)$, i.e., Lemma \ref{2dd}, we have $\mathcal{E}_{\phi}(f)=\mathcal{E}_{\phi}(f^{\star})$.

 On the other hand, suppose that $f(x)=f^{\star}(Ax+x_0)$ for $\mathcal{L}^n$-a.e. $x\in\Omega$ is not established for any $A\in SL(n)$ and $x_0\in\mathbb{R}^n$, then by Lemma \ref{16b}, there exist $u_0\in S^{n-1}$ and $\delta>0$ such that  for any $u\in B(u_0,\delta)\cap S^{n-1}$,
\begin{eqnarray}\label{6m}
 {\rm either\;(i)\; or\;(ii)\; in\; Lemma\;\ref{16b}\; is\; not\; established}.
\end{eqnarray}
Then there exists some $\bar{u}\in B(u_0,\delta)\cap S^{n-1}\cap T_1$, where $T_1$ is given as in Lemma \ref{L4}, such that
\begin{eqnarray}\label{5d}
\mathcal{L}^n(D_{\bar{u}})=0.
\end{eqnarray}
Otherwise, if $\mathcal{L}^n(D_u)>0$ for any $u\in B(u_0,\delta)\cap S^{n-1}\cap T_1$. Thus by Lemma \ref{7g}, there exist $n$ linearly independent vectors $u_1,u_2,\dots,u_n\in B(u_0,\delta)\cap S^{n-1}\cap T_1$ such that $\mathcal{L}^{n}\left(\bigcap_{i=1}^{n}D_{u_i}\right)>0$.
Let $\bar{D}=\bigcap_{i=1}^{n}D_{u_i}$, then for any $x\in \bar{D}$, $u_i\cdot \nabla f(x)=\nabla_{u_i}f(x)=0$ for $i=1,2,\dots,n$.  Since $u_1,u_2,\dots,u_n$ are linearly independent, $\nabla f(x)=0$ for any $x\in \bar{D}$, which is contradictory with (\ref{2z}).

For $\bar{u}$ as in (\ref{5d}), let $f_1=f_{\bar{u}}^s$, by Theorem \ref{2.1}, we have
$\mathcal{E}_{\phi}(f)>\mathcal{E}_{\phi}(f_1)$. Otherwise, if $\mathcal{E}_{\phi}(f)=\mathcal{E}_{\phi}(f_1)$, then by (\ref{5d}), $\bar{u}\in T_1$ and Theorem \ref{2.3}, (\ref{1b}) is established with respect to the direction $\bar{u}$. Thus, both (i) and (ii) in Lemma \ref{16b} are established for this $\bar{u}$,  which is contradictory with (\ref{6m}).

By Remark \ref{19b}, there exists a sequence of  directions $\{u_i\}$, $i=1,2,\cdots$,  such that the sequence defined by $f_{i+1}=f^s_{i,u_i}$ converges to $f^{\star}$  weakly in  $W^{1,1}$.
 Theorem \ref{2.1} and  Lemma \ref{2x} can now be used to conclude that
$$\mathcal{E}_{\phi}(f)>\mathcal{E}_{\phi}(f_1)\geq\cdots\geq \mathcal{E}_{\phi}(f_i)\;{\rm and}\;\lim_{i\rightarrow \infty}\mathcal{E}_{\phi}(f_i)\geq \mathcal{E}_{\phi}(f^{\star}).$$
 Therefore, $\mathcal {E}_{\phi}(f)>\mathcal {E}_{\phi}(f^{\star})$. \qed

\noindent
\small  School of Mathematics and Statistics, Chongqing Technology and Business University,\\
\small  Chongqing 400067, PR China\\
\small E-mail address: yjl@ctbu.edu.cn\\

\noindent
\small  Department of Mathematics, Shanghai University, Shanghai 200444, China; School of Mathematical Sciences, Fudan University, Shanghai 200433, China, \\
\small  Shanghai 200444, PR China\\
\small E-mail address: dongmeng.xi@live.com\\


\begin{thebibliography}{MRY90}


\bibitem{AFP00} L. Ambrosio, N. Fusco, D. Pallara, {\it Functions of Bounded Variation and Free Discontinuity
Problems}, Oxford University Press, Oxford, 2000.

\bibitem{BCFP} M. Barchiesi, G.M. Capriani, N. Fusco, G. Pisante, {\it  Stability of P\'olya-Szeg\"o inequality for log-concave functions}, J. Funct. Anal. 267 (2014), 2264-2297.

\bibitem{Bianchi17} G. Bianchi, R. J. Gardner, P. Gronchi, {\it Symmetrization in geometry}, Adv. Math. 306 (2017), 51-88.

\bibitem{Boroczky13} K.J. B\"or\"oczky, {\it Stronger versions of the Orlicz-Petty projection inequality}, J. Differential Geom. 95 (2013), 215-247.

\bibitem{BLYZ12} K.J. B\"or\"oczky, E. Lutwak, D. Yang, G. Zhang, {\it The log-Brunn-Minkowski inequality}, Adv. Math. 231 (2012), 1974-1997.

\bibitem{BZ88} J.E. Brothers, W.P. Ziemer, {\it Minimal rearrangements of Sobolev functions}, J. Reine Angew. Math. 384 (1988), 153-179.

\bibitem{Bu96} A. Burchard, {\it Cases of equality in the Riesz rearrangement inequality}, Ann. Math. 143 (1996), 499-527.

\bibitem{Bu97} A. Burchard, {\it Steiner symmetrization is continuous in $W^{1,p}$}, Geom.
Funct. Anal. 7 (1997), 823-860.

\bibitem{Bu04} A. Burchard, Y. Guo, {\it Compactness via symmetrization}, J. Funct. Anal. 214 (2004), 40-73.

\bibitem{BF15} A. Burchard, A. Ferone, {\it On the Extremals of the P\'olya-Szeg\"o Inequality}, Indiana Univ. Math. J. 64 (2015), 1447-1463.

\bibitem{Capriani14} G.M. Capriani, {\it The Steiner rearrangement in any codimension}, Calc. Var. Partial Differential Equations 49 (2014), 517-548.


\bibitem{CCF05}  M. Chleb\'{i}k, A. Cianchi, N. Fusco, {\it The perimeter inequality under
Steiner symmetrization: cases of equality}, Ann. of Math. 162
(2005), 525-555.

\bibitem{Ci00} A. Cianchi, {\it Second-order derivatives and rearrangements}, Duke Math. J. 105 (2000), 355-385.

\bibitem{Ci10} A. Cianchi,  {\it On some aspects of the theory of Orlicz-Sobolev spaces}. In {\it  Around the research of Vladimir Maz'ya. I,} volume 11 of {\it  Int. Math. Ser. (N. Y.)},  81-104. Springer, New York, 2010.


\bibitem{CEFT08} A. Cianchi, L. Esposito, N. Fusco, C. Trombetti, {\it A quantitative P\'olya-Szeg\"o principle}, J. Reine
Angew. Math. 614 (2008), 153-189.

\bibitem{CF02} A. Cianchi, N. Fusco, {\it Functions of bounded variation and rearrangements}, Arch. Ration. Mech. Anal. 165 (2002), 1-40.

\bibitem{CF06} A. Cianchi, N. Fusco, {\it Steiner symmetric extremals in P\'olya-Szeg\"o type inequalities},  Adv. Math. 203 (2006), 673-728.

\bibitem{CF0602} A. Cianchi, N. Fusco, {\it Minimal rearrangements, strict convexity and critical points}, Appl. Anal. 85 (2006), 67-85.


\bibitem{CLYZ09} A. Cianchi, E. Lutwak, D. Yang, G. Zhang, {\it Affine Moser-Trudinger and Morrey-Sobolev inequalities}, Calc. Var. Partial Differential Equations 36 (2009), 419-436.

\bibitem{CPS15} A. Cianchi, L. Pick, L. Slav\'ikov\'a {\it Higher-order Sobolev embeddings and isoperimetric inequalities}, Adv. Math. 273 (2015), 568-650.

\bibitem{CFNT17} A. Cianchi, V. Ferone, C. Nitsch, C. Trombetti, {\it Balls minimize trace constants in BV}, J. Reine Angew. Math. 725 (2017), 41-61.


\bibitem{ET04} L. Esposito, C. Trombetti, {\it Convex symmetrization and P\'olya-Szeg\"o inequality}, Nonlinear Anal. 56 (2004), 43-62.


\bibitem{ER09} L. Esposito, P. Ronca, {\it Quantitative P\'olya-Szeg\"o principle for convex symmetrization}, Manuscripta Math. 130
(2009), 339-362.

\bibitem{EG92} L.C. Evans, R.F. Gariepy, {\it Measure Theory and Fine Properties of Functions}, Studies in
Advanced Math., CRC Press, 1992.


\bibitem{Evans90} L.C. Evans, {\it Weak Convergence Methods for Nonlinear Partial Differential Equations},
CBMS Regional Conference Series in Mathematics 74, American Mathematical Society,
Providence, RI, 1990.

\bibitem{FV04} A. Ferone, R. Volpicelli, {\it Convex rearrangement: equality cases in the P\'olya-Szeg\"o
inequality}, Calc. Var. Partial Differential Equations 21 (2004), 259-272.


\bibitem{FMP08} N. Fusco, F. Maggi, A. Pratelli, {\it The sharp quantitative isoperimetric inequality},  Ann. of Math. 168 (2008), 941-980.

\bibitem{Gardner02} R.J. Gardner, {\it The Brunn-Minkowski inequality}, Bull. Amer. Math. Soc. 39 (2002), 355-405.

\bibitem{Ga06} R.J. Gardner, {\it Geometric Tomography}, second edition, Cambridge University
Press, New York, 2006.

\bibitem{GHW14} R.J. Gardner, D. Hug, W. Weil, {\it The Orlicz-Brunn-Minkowski theory: a general framework, additions, and inequalities}, J. Differential Geom. 97 (2014), 427-476.


\bibitem{GHW15} R.J. Gardner, D. Hug, W. Weil, D. Ye, {\it The dual Orlicz-Brunn-Minkowski theory}, J. Math. Anal. Appl. 430 (2015), 810-829.

\bibitem{GZ98} R.J. Gardner, G. Zhang, {\it Affine inequalities and radial mean bodies}, Amer. J. Math. 120 (1998), 505-528.

\bibitem{GMS98}    M. Giaquinta, G. Modica, J. Sou$\check{c}$ek, Cartesian Currents in the Calculus of Variations, Part I: Cartesian Currents, Part II: Variational Integrals, Springer, Berlin, 1998.

\bibitem{Gruber07} P.M. Gruber, {\it Convex and Discrete Geometry}, Springer, Berlin, 2007.


\bibitem{HS09} C. Haberl, F.E. Schuster, {\it General $L_p$ affine isoperimetric inequalities}, J. Differential Geom. 83 (2009), 1-26.

\bibitem{HLYZ10} C. Haberl, E. Lutwak, D. Yang, G. Zhang, {\it The even Orlicz Minkowski problem}, Adv. Math. 224 (2010), 2485-2510.


\bibitem{HP14} C. Haberl, L. Parapatits, {\it The centro-affine Hadwiger theorem}, J. Amer. Math. Soc. 27 (2014), 685-705.

 \bibitem{HS0902} C. Haberl, F.E. Schuster, {\it Asymmetric affine $L_p$ Sobolev inequalities}, J. Funct. Anal. 257 (2009), 641-658.

\bibitem{HSX12} C. Haberl, F.E. Schuster, J. Xiao, {\it An asymmetric affine P\'olya-Szeg\"o principle}, Math. Ann. 352 (2012), 517-542.

\bibitem{Kawohl86} B. Kawohl, {\it On the isoperimetric nature of a rearrangement inequality and its consequences for
some variational problems}, Arch. Rat. Mech. Anal. 94 (1986), 227-243.

\bibitem{LX17} A.-J. Li, D. Xi, G. Zhang, {\it Volume inequalities of convex bodies from cosine transforms on Grassmann manifolds}, Adv. Math. 304 (2017), 494-538.

\bibitem{Lin17} Y. Lin, {\it Smoothness of the Steiner symmetrization},  Proc. Amer. Math. Soc. 146 (2018), 345-357.

\bibitem{LYJ17} Y. Lin, {\it Affine Orlicz P\'olya-Szeg\"o principle for log-concave functions}, J. Funct. Anal. 273 (2017), 3295-3326.


\bibitem{LL10} Y. Lin, G. Leng, {\it Convex bodies with minimal volume product in $\mathbb{R}^2$-a new proof}, Discrete Math. 310 (2010), 3018-3025.

\bibitem{Ludwig10} M. Ludwig, {\it General affine surface areas}, Adv. Math. 224 (2010), 2346-2360.

\bibitem{LR99} M. Ludwig, M. Reitzner, {\it A characterization of affine surface area}, Adv. Math. 147 (1999),  138-172.

\bibitem{LR10} M. Ludwig, M. Reitzner, {\it A classification of $SL(n)$ invariant valuations}, Ann. of Math. 172 (2010), 1219-1267.

\bibitem{LXZ11} M. Ludwig, J. Xiao, G. Zhang, {\it Sharp convex Lorentz-Sobolev inequalities}, Math. Ann. 350 (2011), 169-197.

\bibitem{Lu93} E. Lutwak, {\it The Brunn-Minkowski-Firey theory. I. Mixed volumes and the Minkowski problem}, J. Differential Geom. 38 (1993), 131-150.

\bibitem{Lu96} E. Lutwak, {\it The Brunn-Minkowski-Firey theory. II. Affine and geominimal surface areas}, Adv. Math. 118 (1996), 244-294.

\bibitem{LYZ00} E. Lutwak, D. Yang, G. Zhang, {\it $L_p$ affine isoperimetric inequalities}, J. Differential Geom. 56 (2000), 111-132.

\bibitem{LYZ02} E. Lutwak, D. Yang, G. Zhang, {\it Sharp affine $L_p$ Sobolev inequalities}, J. Differential Geom. 62 (2002), 17-38.

\bibitem{LYZ04} E. Lutwak, D. Yang, G. Zhang, {\it On the $L_p$-Minkowski problem}, Trans. Amer. Math. Soc. 356 (2004), 4359-4370.

\bibitem{LYZ06} E. Lutwak, D. Yang, G. Zhang, {\it Optimal Sobolev norms and the $L^p$ Minkowski problem}, Int. Math. Res. Not. (2006), Art. ID 62987, 1-21.

\bibitem{LYZ10} E. Lutwak, D. Yang, G. Zhang, {\it Orlicz projection bodies}, Adv. Math. 223 (2010), 220-242.

\bibitem{LYZ1002} E. Lutwak, D. Yang, G. Zhang, {\it Orlicz centroid bodies}, J. Differential Geom. 84 (2010), 365-387.

\bibitem{Ma85} V.G. Ma${\rm z}^{\prime}$ya, {\it Sobolev Spaces}, Springer-Verlag, Berlin (1985).

\bibitem{Ma11} V.G. Ma${\rm z}^{\prime}$ya, {\it Sobolev Spaces with Applications to Elliptic Partial Differential Equations}, Springer, Heidelberg, 2011.

\bibitem{Nguyen16} V.H. Nguyen, {\it New approach to the affine P\'olya-Szeg\"o  principle and the stability version of the affine Sobolev inequality}, Adv. Math. 302 (2016), 1080-1110.

\bibitem{PS51} G. P\'olya, G. Szeg\"o, {Isoperimetric inequalities in mathematical physics}, Ann. Math. Stud. 27, Princeton University Press (1951).


\bibitem{Schneider13} R. Schneider, {\it Convex Bodies: The Brunn-Minkowski Theory}, Encyclopedia of Mathematics and its Applications, 151. Cambridge Univ. Press, Cambridge, 2014.

\bibitem{Ta94} G. Talenti, {\it Inequalities in rearrangement invariant function spaces}, In: M. Krbec, A. Kufner, B. Opic, J. R\'akosnik (eds.) Nonlinear Analysis, Function Spaces and Applications, vol. 5, 177-230, Prometheus, Prague (1994).

\bibitem{Trudinger97} N.S. Trudinger, {\it On new isoperimetric inequalities and symmetrization}, J. reine angew. Math. 488 (1997), 203-220.


\bibitem{Vol67} A.I. Vo${\rm l}^{\prime}$pert, {\it Spaces $BV$ and quasi-linear equations}, Math. USSR Sb. 17 (1967), 225-267.

\bibitem{Wang13} T. Wang, {\it The affine P\'olya-Szeg\"o principle:
Equality cases and stability}, J. Funct. Anal. 265 (2013), 1728-1748.

\bibitem{XGL} D. Xi, L. Guo, G. Leng, {\it Affine inequalities for $L_p$ mean zonoids}, Bull. Lond. Math. Soc. 46 (2014), 367-378.

\bibitem{XJL14} D. Xi, H. Jin, G. Leng, {\it The Orlicz Brunn-Minkowski inequality}, Adv. Math. 260 (2014), 350-374.

\bibitem{XL16} D. Xi, G. Leng, {\it Dar's conjecture and the log-Brunn-Minkowski inequality}, J. Differential Geom. 103 (2016), 145-189.

\bibitem{Zhang99} G. Zhang, {\it The affine Sobolev inequality}, J. Differential Geom. 53 (1999), 183-202.

\bibitem{ZZX14} B. Zhu, J. Zhou,  W. Xu, {\it Dual Orlicz-Brunn-Minkowski theory}, Adv. Math. 264 (2014), 700-725.

\bibitem{Zhu12} G. Zhu, {\it The Orlicz centroid inequality for star bodies}, Adv. in App. Math. 48 (2012), 432-445.

\bibitem{Ziemer89} W.P. Ziemer, {\it Weakly Differentiable Functions: Sobolev spaces and functions of bounded variation}, GTM 120, Springer Verlag, 1989.
\end{thebibliography}
\end{document}